\documentclass[12pt]{amsart}
\usepackage{macros}
\usepackage{comment}

\begin{document}
\title
[Experiments on the Brauer map in high codimension]
{Experiments on the Brauer map in high codimension}

\author{Siddharth Mathur}
\address{Mathematisches Institut\\Heinrich-Heine-Universit\"at\\40204 D\"usseldorf, Germany.}
\email{siddharth.p.mathur@gmail.com}

\begin{abstract}

\noindent Using twisted sheaves, formal-local methods, and elementary transformations we show that separated algebras spaces which are constructed as pushouts or contractions (of curves) have enough Azumaya algebras. This implies: (1) Under mild hypothesis, every cohomological Brauer class is representable by an Azumaya algebra away from a closed subset of codimension $\geq 3$, generalizing an early result of Grothendieck and, (2) $\mathrm{Br}(X)=\mathrm{Br}'(X)$ when $X$ is an algebraic space obtained from a quasi-projective scheme by contracting a curve. This result is valid for all dimensions but if we specialize to surfaces, it solves the question entirely: there are always enough Azumaya algebras on separated surfaces.

\end{abstract}

	\maketitle
	\tableofcontents

\section{Introduction}
In his original treatises on the Brauer group \cite{Brauer2}, Grothendieck posed the question: does every torsion class in $\H^2(X, \mathbf{G}_m)$ arise from an Azumaya algebra or a Brauer-Severi scheme? That is, do the geometric, $\Br(X)$, and cohomological, $\Br'(X)$, Brauer groups agree? One can show that $\Br(X) \subset \Br'(X)$ always holds, so the difficulty lies in showing surjectivity. This has proven to be quite a challenge and remains an open problem in general.

Some of the early progress is due to Grothendieck: he showed that on any regular noetherian scheme, every cohomological Brauer class is representable by an Azumaya algebra away from a set of codimension $\geq 3$. Another early result is due to Hoobler, who proved that there are enough Azumaya algebras on abelian varieties (see \cite{hoobler1972brauer}). However, the result with perhaps the most utility is an (unpublished) Theorem of Gabber: any cohomological Brauer class on a scheme $X$ is geometric as long as $X$ admits an ample line bundle (see \cite{dejongample} for a different proof due to de Jong). 

As an important precursor to this result, Gabber proved that any scheme which is the separated union of two affine open subschemes always has enough Azumaya algebras (see Gabber's thesis \cite[Chapter 2, Theorem 1]{gabber1981some}). His ideas refine an argument of Hoobler, who proved a similar result but with finiteness and regularity hypothesis (see \cite{hoobler1980cohomological}). Extending these arguments (or finding a counterexample) is an important problem and a general solution has remained out of reach for almost 50 years. Outside of the quasi-projective setting very little is known and it is difficult to know what to expect. Indeed, consider the contrasting results:
\begin{enumerate}
\item If $X$ is a two-dimensional algebraic space which is geometrically normal, separated and finite-type over a field $k$ then $\mathrm{Br}(X)=\mathrm{Br}'(X)$ (Schr\"oer \cite{enoughAzumaya}) 
\item There is a (non-separated) normal two-dimensional scheme which is finite-type over a field $k$ with $\Br(X) \neq \Br'(X)$. (Edidin, Hassett, Kresch and Vistoli \cite{EHKV})
\end{enumerate} 

\noindent To appreciate how little we know, note that the question is open even when $M$ is a smooth separated threefold over the complex numbers.

Using the K-theory of commutative rings and formal glueing techniques, we are able to refine the result of Gabber and Hoobler:

\begin{Theorem} \label{T1tubular} Let $X$ be an algebraic space with affine diagonal. Suppose that $X$ admits an open immersion $j: \Spec A \to X$ in such a way that there is an affine flat neighborhood $f: \Spec B=\widehat{Z} \to X$ of $Z=(X \backslash \Spec A)_{\mathrm{red}}$, then $\Br(X)=\Br'(X)$. 
\end{Theorem}

Indeed, any scheme which is the separated union of two affine open subschemes satisfies the hypothesis of the aforementioned theorem. Then, by using the theory of Tannaka duality (as in \cite{hall2019coherent}) one can readily produce affine flat neighborhoods and as a corollary, we are also able to improve Grothendieck's early result:

\begin{corollary} \label{T1} Let $X$ be a separated algebraic space which is noetherian and has a dense open subset consisting of regular points (e.g. a generically reduced variety). Then, given any cohomological Brauer class $\alpha \in \mathrm{Br}'(X)$ there is an open subset $U \subset X$ with $\codim X \backslash U \geq 3$ such that $\alpha|_U \in \mathrm{Br}(U)$ is representable by an Azumaya algebra. 
\end{corollary}

If there is any chance of finding a cohomology class which isn't geometric one needs to consider algebraic spaces which are not quasi-projective. Moreover, since we now know Brauer classes are generically geometric we would like to study their geometricity near closed points of high codimension. Thus, we need to find higher dimensional examples of algebraic spaces which are not quasi-projective but which also have plenty of cohomological Brauer classes. Towards this end, we study algebraic spaces $X$ which arise from quasi-projective schemes $Y$ by contracting subvarieties $S \subset Y$ to a point. Constructing contractions in the category of schemes is a very difficult problem, but by results of Artin (see \cite{artincontraction}) one may always contract \emph{rigid} subvarieties in the category of algebraic spaces. 

In fact, not only do such contractions rarely have ample line bundles, they are rarely schemes. We will show that torsion line bundles defined on an \'etale tubular neighborhood of $S$ give rise to cohomological Brauer classes on $X$, which are also often nontrivial. To ensure there are plenty of torsion line bundles on an \'etale tubular neighborhood of $S$ we will consider the case when $S$ is a $1$-dimensional scheme. Indeed, in this setting every line bundle on $S$ extends to a \'etale tubular neighborhood. Using this strategy, we will give examples of non-schematic algebraic spaces in all dimensions $\geq 3$ whose cohomological Brauer group is not finitely generated (see Proposition \ref{examples}).

Unfortunately, this strategy will not produce counterexamples. In fact, we are able to prove that the geometric and cohomological Brauer groups must coincide in this setting:

\begin{Theorem} \label{T2} Let $X$ be a separated algebraic space which is of finite-type over a field $k$. If there is a proper morphism $f: Y \to X$ such that
\begin{enumerate}
\item $f$ is an isomorphism away from finitely many points, with
\item $Y$ quasi-projective over $k$, and
\item the exceptional locus of $f$ is $1$-dimensional.
\end{enumerate}
then $\mathrm{Br}(X)=\mathrm{Br}'(X)$.
\end{Theorem}

This result yields the first examples of algebraic spaces $X$ in all dimensions which $(1)$ are not schemes, and $(2)$ satisfy $\mathrm{Br}(X)=\mathrm{Br}'(X)$. Even if we specialize to the case when $\dim X=2$ we obtain a generalization of Schr\"oer's result:

\begin{corollary} \label{surfaces} Let $X$ be a two-dimensional algebraic space which is of finite-type and separated over a field $k$, then $\mathrm{Br}(X)=\mathrm{Br}'(X)$.
\end{corollary}

\noindent thereby settling the question for separated surfaces.
\\

\noindent \textbf{Acknowledgments} I am indebted to my advisor Max Lieblich for suggesting this problem to me and for his careful guidance. Many thanks go to Jarod Alper and Aise Johan de Jong for their helpful suggestions and encouragement. I am grateful to Jack Hall for carefully reading an earlier version of this paper and for his comments and ideas, Lemma \ref{primetopthick} is essentially due to him. I would like to thank Robert Lazarsfeld for suggesting the example of hypersurfaces. I also benefitted from conversations with Lucas Braune, Andrew Kresch, Jake Levinson, David Rydh, Stefan Schr\"oer, Minseon Shin, and Angelo Vistoli. I would also like to thank the referee for carefully reading an earlier version of this article and for their helpful suggestions. A part of this research was conducted in the framework of the research training group \emph{GRK 2240: Algebro-geometric Methods in Algebra, Arithmetic and Topology}, which is funded by the DFG.
\section{The Surjectivity of the Brauer Map}

\subsection{Grothendieck's Question}
We begin by briefly reviewing the theory of the Brauer group with a focus on its relationship to gerbes and twisted sheaves. Let $X$ denote an algebraic stack. We work throughout with the (big) fppf site on $X$.

\begin{definition} \label{DefCohoBrauer} The \emph{cohomological Brauer group} of $X$ is $\H^2(X,\mathbf{G}_m)_{\text{tors}}$, we denote it by $\Br'(X)$.
\end{definition}

\begin{proposition} \label{PGLpres} For each $n \geq 2$, the two rows of the commutative diagram form exact sequences in the category of sheaves of groups on $X$

\[\begin{tikzcd}
0 \arrow{r} & \mathbf{\mu}_n \arrow{r} \arrow{d}{\ker_n} & \mathbf{SL}_n \arrow{r} \arrow{d} & \mathbf{PGL}_n \arrow{r} \arrow{d}{\mathrm{id}} & 0 \\
0 \arrow{r} & \mathbf{G}_m \arrow{r} & \mathbf{GL}_n \arrow{r} & \mathbf{PGL}_n \arrow {r} & 0\\
\end{tikzcd}
\]
where $\ker_n$ denotes the kernel of the $n$th power map $\mathbf{G}_m \to \mathbf{G}_m$.
\end{proposition}

\begin{proof} See \cite{giraud} pg. 341-343. 
\end{proof}

\noindent We follow Giraud's nonabelian formalism \cite{giraud}: let $\H^1(X, G)$ denote isomorphism classes of left $G$-torsors and let $\H^2(X, \mathbf{G}_m)$ denote isomorphism classes of $\mathbf{G}_m$-gerbes. The proposition above implies we obtain a morphism of two long exact sequences (up to degree 2) of pointed sets for any $n$:

\[\begin{tikzcd}
... \arrow{r} & \H^1(X,\mathbf{\mu}_n) \arrow{r} \arrow{d} & \H^1(X,\mathbf{SL}_n)\arrow{r} \arrow{d} & \H^1(X,\mathbf{PGL}_n) \arrow{r} \arrow{d}{\mathrm{id}} & \H^2(X, \mathbf{\mu}_n) \arrow{r} \arrow{d} & ... \\
... \arrow{r} & \H^1(X,\mathbf{G}_m) \arrow{r} &  \H^1(X,\mathbf{GL}_n) \arrow{r} & \H^1(X,\mathbf{PGL}_n) \arrow{r}{\delta_n} & \H^2(X,\mathbf{G}_m) \arrow{r} & ...\\
\end{tikzcd}
\]
By commutativity we note that every $\mathbf{PGL}_n$-torsor gets mapped to a $n$-torsion class in $\H^2(X, \mathbf{G}_m)$. In particular, we have $\text{Im}(\delta_n) \subset \Br'(X)$. Recall that $\mathbf{PGL}_n$-torsors correspond to geometric objects such as Azumaya algebras of degree $n$ or Brauer-Severi varieties of relative dimension $n-1$. As such, we make the following definition

\begin{definition} \label{def:geombr} Let $X$ be an algebraic stack, we call the union $\bigcup_{n \geq 2} \text{Im}(\delta_n) \subset \Br'(X)$ the \emph{geometric Brauer group}. We denote it by $\Br(X)$.
\end{definition}

Note that by our definition, it is not clear that the set $\Br(X)$ is a $\emph{group}$, but in fact it is and for a proof we refer to the following 

\begin{Theorem} \label{BrauerMap} Let $X$ be an algebraic stack. Then there is a natural injection of abelian groups $\delta: \Br(X) \to \Br'(X)$, the Brauer map. Under this injection, the Brauer-class of an Azumaya Algebra $A$ of degree $n$ can be identified with the $\mathbf{G}_m$-gerbe of trivializations of $A$, $\mathscr{X}_A$. More precisely, $\mathscr{X}_A(T)$ is the category of pairs $(V, \phi)$ where $V$ is a vector bundle on $T$ and a trivialization $\phi: \End(V) \cong A|_T$. A morphism from $(V,\phi)$ to $(W, \psi)$ is an isomorphism from $V \to W$ compatible with the trivializations $\phi$ and $\psi$.
\end{Theorem}
\begin{proof} See \cite[Chapter V.4.4]{giraud}.
\end{proof}

\noindent $\mathbf{Question}$ (Grothendieck \cite{Brauer2}) Is the Brauer map $\delta: \Br(X) \to \Br'(X)$ surjective?

\subsection{Twisted Sheaves}
 In what follows, we define twisted sheaves. At the expense of introducing stacky complexity, they will clarify the question above. Let $D \subset \mathbf{G}_m$ be a subgroup scheme.

\begin{definition} \label{DefTwisted} Let $\mathscr{X}$ be a $D$-gerbe on $X$. Any sheaf $F$ on $\mathscr{X}$ has a right action it inherits from the left action of $\text{Aut}(x) \cong D$ on $x \in \mathscr{X}$: if $\phi: x \to x$ is an automorphism then so is $F(\phi):F(x) \to F(x)$. If we assume that $F$ is also quasicoherent, then it also has a left action of $D$ it obtains via the $\mathcal{O}$-module structure. Indeed, a (local) section of $D \subset \mathbf{G}_m \subset \mathcal{O}_{\ms X}$ induces a (local) automorphism of $F$ defined by left-multiplication. We say a quasi-coherent sheaf $F$ on $\mathscr{X}$ is \emph{twisted} if this left action is the one associated to the right action above.
\end{definition}

By the Kummer sequence, every torsion class in $\H^2(X, \mathbf{G}_m)$ has a preimage in $\H^2(X, \mu_n)$. We say that a $\mu_n$-gerbe $\ms Y$ over $X$ \emph{gives rise} to a torsion $\mathbf{G}_m$-gerbe $\ms X$ over $X$ if $[\ms Y]$ is a preimage of $[\ms X]$. Next we explain the connection between twisted vector bundles and Azumaya algebras on algebraic spaces. The following result is well known to the experts but not stated or proven in this generality (compare with \cite[Proposition 3.1.2.1 (1)]{lieblich2008twisted} and \cite[Theorem 3.6]{EHKV}).

\begin{proposition} \label{thm:brequiv} Let $X$ be an algebraic space, $\ms X$ a $\mathbf{G}_m$-gerbe over $X$, and $\ms Y$ a $\mu_n$-gerbe over $X$ giving rise to $\ms X$. Then the following are equivalent
\begin{enumerate}
\item The cohomological Brauer class $[\ms X]$ is in the geometric Brauer group $\Br(X)$.
\item The $\mathbf{G}_m$-gerbe $\ms X$ admits a nonzero twisted vector bundle.
\item The $\mu_n$-gerbe $\ms Y$ admits a nonzero twisted vector bundle.
\end{enumerate}
\end{proposition}

\begin{proof} (1 $\Rightarrow$ 2) Let $T$ be any algebraic space over $X$, as we saw in Theorem \ref{BrauerMap} we have that $\ms X(T)$ is the groupoid of pairs $(V, \phi)$ where $V$ is a vector bundle on $T$ and $\phi$ is an isomorphism $\End(V) \to A|_T$. There is a tautological vector bundle $\mathbf{V}$ on this stack: it is the sheaf $(V,\phi) \mapsto \Gamma(T, V)$ where the restriction maps are those induced by morphisms in $\mathscr{X}_A$. Noting that the automorphisms of $(V, \phi)$ are precisely left multiplication on $V$ by an element in $\mathbf{G}_m(T)$, it follows that $\mathbf{V}$ is a twisted vector bundle of nonzero constant rank on $\ms X$.

(2 $\Rightarrow$ 1) Let $V$ be a twisted vector bundle of nonzero constant rank and consider $\text{End}(V)$. Note that $\mathbf{G}_m$ acts trivially on $V \otimes V^{\vee} \cong \text{End}(V)$ and thus, it is the pull-back of an algebra on $X$ (even though $V$ is \emph{not} the pull-back of a vector bundle on $X$). In fact, $\text{End}(V)$ is the pullback of an \emph{Azumaya} algebra on $X$. That is to say, after an \'etale cover $X' \to X$, there is an isomorphism $\text{End}(V)|_{X'} \cong \text{End}(W)$ for a vector bundle $W$ on $X'$. Indeed, \'etale locally on $X$, the gerbe map $\mathscr{X} \to X$ admits a section which, in turn, yields an equivalence $\mathscr{X} \to \B\mathbf{G}_m$. This equivalence endows $\mathscr{X}$ with a twisted line bundle, $L$, which, in turn, gives rise to a trivialization of $\text{End}(V) \cong \text{End}(V \otimes L^{\vee})$. Note that the inertia of $\ms X$ acts trivially on $V \otimes L^{\vee}$, so that it is the pullback of a vector bundle on $X$. This shows that, any $T$ point of $\mathscr{X}$ canonically yields a trivialization of $\text{End}(V)$. In other words we obtain a map $\mathscr{X} \to \mathscr{X}_{\text{End}(V)}$ which respects the $\mathbf{G}_m$-gerbe structure. Since this must be an isomorphism, the first statement now follows from Theorem \ref{BrauerMap}.

Suppose $\ms Y$ is a $\mu_n$-gerbe on $X$ which gives rise to $\ms X$. By the nonabelian formalism of Giraud, this implies there exists a $\mu_n$ gerbe $\mathscr{Y}$ and a map $i: \mathscr{Y} \to \mathscr{X}$ equivariant for $\text{ker}_n: \mu_n \to \mathbf{G}_m$ (see \cite[Corollaire 3.1.8 pg. 251]{giraud} for a proof). 

(2 $\Rightarrow$ 3) Thus, if $V$ is a twisted vector bundle on $\ms X$, we may pull back a given twisted vector bundle to obtain one over $\mathscr{Y}$. 

(3 $\Rightarrow$ 2) Conversely, suppose that $W$ is a nonzero twisted vector bundle on $\ms Y$, then we claim that $i_*W$ has a nonzero twisted subbundle of finite rank. Indeed, since $i_*W$ inherits an action by the band $\mathbf{G}_m$ of $\ms X$ we obtain an eigendecomposition $i_*W=\bigoplus_{i \in \mathbf{Z}} (i_*W)_i$. Our claim is that the twisted part $(i_*W)_1$ is locally free of rank $n$. We can check this locally, so we may assume that $X=\Spec R$ is the spectrum of a strictly local ring, $\ms Y=\mathbf{B}\mathbf{\mu}_{n, R}$, $\ms X=\mathbf{B}\mathbf{G}_{m,R}$, $i$ is the standard morphism $\B\mathbf{\mu}_{n,R} \to \mathbf{B}\mathbf{G}_{m,R}$ and $W=\chi^{\oplus n}$ where $\chi \in \Pic(\mathbf{B}\mu_{n,R})$ is the twisted line bundle. Let $\pi: \mathbf{B}\mathbf{G}_{m,R} \to \Spec R$ denote the gerbe map and note that if $\chi'$ is a twisted line bundle on $\mathbf{B}\mathbf{G}_{m,R}$ then it suffices to show $\chi' \otimes (i_*W)_1^{\vee}$ is a locally free module of rank $n$ on $\Spec R$. Indeed, we have
\[\chi' \otimes (i_*W)_1^{\vee}=\pi_*(\chi' \otimes (i_*W)^{\vee})=\widetilde{\H^0(\mathbf{BG}_{m,R}, \mathbf{Hom}(\chi', i_*W))}=\Hom(\chi',i_*W)\]
and therefore it suffices to show $\Hom(\chi', i_*W)$ is a free $R$-module of finite rank. However, by adjunction applied to $i$ we obtain
\[\Hom(\chi',i_*W)=\Hom(i^*\chi',W)=\Hom(\chi, \chi^{\oplus n})\]
which is a free $R$-module of rank $n$. 
\end{proof}

Thus, to show that $\Br(X)=\Br'(X)$ it suffices to show that any $\mu_n$-gerbe over $X$ carries a nonzero twisted vector bundle. In this way, we can circumvent the use of Azumaya algebras altogether and focus on producing twisted vector bundles on gerbes instead.

\section{There are enough Azumaya Algebras in codimension 2}

\subsection{Formal affine neighborhoods}

Algebraic spaces do not admit affine open neighborhoods about points (or closed subsets) in general. To remedy this, we show that we can find \emph{flat} neighborhoods of points (and \emph{affine} closed subsets). Even though flat neighborhoods are not subsets of the underlying topological space they will be sufficient for our purposes. 

\begin{definition} \label{def:flathood} Let $X$ be an algebraic stack and $C \subset X$ a finitely presented closed substack. Then we call a flat morphism $Y \to X$ a \emph{flat neighborhood} of $C$ if the natural morphism $Y \times_X C \to C$ is an isomorphism. We say such a neighborhood is \emph{affine} if $C$ is an affine scheme.
\end{definition}

\begin{proof}[\textbf{Proof of Theorem \ref{T1tubular}}] We have the following cartesian square
\[\begin{tikzcd} 
\Spec C \rar \dar & \widehat{Z} \dar{f} \\
\Spec A \rar{j} & X \\
\end{tikzcd} \]
which, using the terminology of \cite{hall2016mayer}, is a \emph{flat Mayer-Vietoris} square by \cite[Lemma 3.2]{hall2016mayer}. 
By Proposition \ref{thm:brequiv}, it remains to show that every $\mu_n$-gerbe over $X$ admits a nonzero twisted vector bundle. Thus, by \cite[Lemma 3.1]{hall2016mayer}, if we base change by a $\mu_n$-gerbe $\ms X \to X$ we obtain another flat Mayer-Vietoris square
\[\begin{tikzcd} 
\ \ms X_C \rar \dar & \ms X_{\widehat{Z}} \dar{f} \\
\ms X_A \rar{i} & \ms X \\
\end{tikzcd} \]
and the natural functor
\[\Qcoh(\ms X) \cong \Qcoh(\ms X_A) \times_{\Qcoh(\ms X_C)} \Qcoh(\ms X_{\widehat{Z}})\] 
is an equivalence of categories (see \cite[Theorem B(1)]{hall2016mayer}). In particular, it suffices to produce a twisted vector bundle on $\ms X_A$, another one on $\ms X_B$ and an isomorphism between their two restrictions on $\ms X_C$. By \cite{gabber1981some} and Proposition \ref{thm:brequiv} there exists nonzero twisted vector bundles $V$ and $W$ on $\ms X_A$ and $\ms X_B$ respectively. By \cite[Corollary 3.1.4.4]{lieblich2008twisted} we may replace $V$ and $W$ with $(V \otimes P_1)^{\oplus k}$ and $(W \otimes P_2)^{\oplus m}$ for suitable projective modules $P_1$ and $P_2$ on $A$ and $B$ respectively so that $V^{\otimes n}$ and $W^{\otimes n}$ are both free modules. By taking appropriate direct sums we may also suppose that $V$ and $W$ have the same rank. 

Now consider the restrictions of $V$ and $W$ to $\ms X_C$ and set $M=V \otimes V^{\vee}$ and $N=W \otimes V^{\vee}$. Observe that $M$ and $N$ are vector bundles on $\Spec C$ satisfying $M^{\otimes n} \cong N^{\otimes n}=\mathcal{O}_{\Spec C}^{\oplus j}$. Thus, we have $N^{\oplus m}\cong M^{\oplus m}=\mathcal{O}_{\Spec C}^{\oplus r}$ for some $m$ by \cite[Proposition 3.1.4.3 (4)]{lieblich2008twisted}. Moreover, there is an obvious isomorphism $V \otimes N \cong W \otimes M$. This implies there are isomorphisms
\[V^{\oplus r} \cong V \otimes N^{\oplus m} \cong W \otimes M^{\oplus m} \cong W^{\oplus r}\]
showing that $V^{\oplus r}$ and $W^{\oplus r}$ agree on $\ms X_C$. \end{proof}

\begin{proposition} \label{prop:flathood} Let $X$ be a quasi-compact algebraic space with affine diagonal and $C \to X$ a finitely presented closed immersion which is affine i.e. $C=\Spec A$. Then there is an affine flat neighborhood $\widehat{C} \to X$ of $C$.
\end{proposition}

\begin{proof} Note that by \cite[Theorem C and D]{RydhApproximation}, we may write $X=\lim_{\lambda \in I} X_{\lambda}$ where each $X_{\lambda}$ is an algebraic space of finite type over $\Spec \mathbf{Z}$ with affine diagonal, each bonding morphism $X_{\lambda} \to X_{\lambda'}$ is affine, and each morphism $X \to X_{\lambda}$ is affine. Moreover, because $C \to X$ is of finitely presention, there is a $\lambda \in I$ and a morphism $C_{\lambda} \to X_{\lambda}$ equipped with an isomorphism $C_{\lambda} \times_{X_{\lambda}} X \cong C$ of spaces over $X$ (see e.g. \cite[Tag 07SK]{stacks}). By enlarging $\lambda$ if necessary, we may also assume that $C_{\lambda}$ is an affine scheme and that $C_{\lambda} \to X_{\lambda}$ is a closed immersion (see \cite[Proposition B.3]{RydhApproximation}). If we can find an affine flat neighborhood $\widehat{C_{\lambda}} \to X_{\lambda}$ of $C_{\lambda}$, then by pulling back along the affine morphism $X \to X_{\lambda}$ we obtain an affine flat neighborhood of $C$. Thus, we may assume that $X$ is of finite type over $\Spec \mathbf{Z}$.

Note that $C=C_0$ is cut out by an ideal sheaf $I_C \subset \mathcal{O}_X$, we define $C_i \subset X$ to be closed subscheme cut out by the ideal sheaf $I_C^{i+1} \subset \mathcal{O}_X$. Note that each $C_i=\Spec A_i$ is affine and that the system of ring maps
\[\dots A_3 \to A_2 \to A_1 \to A_0=A\]
is adic. Moreover, this gives rise to an infinite inductive system of closed immersions of affine schemes
\[C=C_0 \to C_1 \to C_2 \to \dots \to X\]
Thus, set $\widehat{C}=\Spec \widehat{A}$ where $\widehat{A}=\lim_i A_i$ which by \cite[Proposition 7.2.7 and Corollaire 7.2.8]{grothendieck1960elements} is a noetherian adic ring. By the universal property of limits this affine scheme is the colimit of the inductive system in the category of affine schemes. However, since $X$ is not affine we cannot use this to produce a map $\widehat{C} \to X$. For this, we use \cite[Theorem 1.5 (ii)]{hall2019coherent} to obtain a map $\widehat{C} \to X$.

It remains to check that $\widehat{C} \to X$ is a flat neighborhood. This follows because $\widehat{C} \times_X C_i \to C_i$ is an isomorphism for every $i$ (see \cite[Tag 0523]{stacks}). Indeed, by Grothendieck's existence theorem applied to $\widehat{C}$ we have
\[\mathcal{O}_X/I_C^{i+1}|_{\widehat{C}}=\lim_n (j_n)_*(\mathcal{O}_X/I_C^{i+1}|_{C_n})\]
and the left hand side is isomorphic to $(j_i)_*\mathcal{O}_X/I_C^{i+1}$ since it is so for all but finitely many terms in the limit. \end{proof}

\begin{corollary} \label{prop:stratifyBrauer} Let $X$ be an algebraic space with affine diagonal which admits an open immersion $j: \spec(A) \to X$ whose complement $Z=(X\backslash \Spec(A))_{\mathrm{red}}$ is a finitely presented closed subscheme which is affine i.e. $Z=\Spec R$. Then $\Br(X)=\Br'(X)$.
\end{corollary}

\begin{proof} First observe that $X$ is necessarily quasi-compact: if $\{U_i\}$ is an open cover, there is a finite subcollection of the $U_i$ which cover $\Spec A$ \emph{and} $\Spec R$. By Proposition \ref{prop:flathood} there is a flat affine neighborhood of $\Spec R$, call it $f: \Spec B=\widehat{Z} \to X$. Now Theorem \ref{T1tubular} applies. \end{proof}

The next result guarantees that given any finite set of codimension $\leq 1$ points in a separated locally noetherian algebraic space, there exists a affine open subscheme containing this finite set.

 \begin{proposition} \label{CKincodim1} Suppose $X$ is a separated, locally noetherian algebraic space and $x_1,...,x_n$ is a finite set of points all having codimension $\leq 1$. Then there is an affine open $\Spec{B} \subset X$ containing all the $x_i$.
 \end{proposition}

 \begin{proof} By \cite[Tag 0ADD]{stacks} we may find a open subscheme $U \subset X$ containing all the $x_i$. Now, since any finite set of points in a scheme admits a quasi-compact open neighborhood, we may replace $U$ with a noetherian scheme. To finish, apply \cite[Tag 09NN]{stacks}.
 \end{proof}

 \begin{remark} \label{remarkCK} One says an algebraic space satisfies the Chevalley-Kleiman property if every finite set of points admits a common affine open neighborhood. The previous two propositions may be interpreted as saying every separated, locally noetherian algebraic space satisfies the Chevalley-Kleiman property in codimension $\leq 1$.

 \end{remark}

\begin{proof}[\textbf{Proof of Corollary \ref{T1}}]  Consider an arbitrary $\mu_n$-gerbe $\sX$ on $X$. Using the results above we will show there exists a nonzero twisted vector bundle on $\sX|_U$ where $U \subset X$ is a dense open subset containing all codimension $\leq 2$ points. Our strategy will be to find a nonzero twisted vector bundle $V$ on a dense open subspace containing all the $\emph{singular}$ codimension $\leq 2$ points. Given such a $V$, any reflexive extension, $V'$, of this twisted vector bundle will be locally free at $\emph{all}$ codimension $\leq 2$ points. Indeed, it will be locally free at all singular codimension $\leq 2$ points by construction and on the remaining regular points of codimension $\leq 2$, the Auslander-Buchsbaum formula implies $V'$ is locally free. Passing to the open locus where $V$ is locally free yields the result.

Since the regular locus of $X$ is dense and open, the complement, $\text{Sing}(X)$ is a closed subspace of codimension $\geq 1$ and therefore we may write
\[\text{Sing}(X)=C_1 \cup...\cup C_n \cup B_1 \cup ...\cup B_m\]
where the $C_i$ are codimension $1$ components and the $B_i$ are components of higher codimension. Using Proposition \ref{CKincodim1} there exists a dense open affine subscheme $\Spec (A) \subset X$ containing the generic points of each of the $C_i$. This means all but finitely many codimension 2 points of $X$ lying along $\text{Sing}(X)$ are contained in $\Spec (A)$. Indeed, since $\Spec A \cap C_i$ contains the generic point of the noetherian space $C_i$, there are finitely many components in $C_i \backslash (\Spec A \cap C_i)$ all of whose generic points have codimension $\geq 2$ in $X$. Therefore, the union of the generic points of the finitely many components of $C_1 \backslash (\Spec A \cap C_1),..., C_n \backslash (\Spec A \cap C_n)$ is finite and since the $B_i$ are already codimension $\geq 2$ they contribute at most one codimension $2$ point each.

Now, since $\Spec A$ is dense the (finitely many) codimension $2$ points of $X$ which lie on $\text{Sing}(X) \cap (X \backslash \Spec A)$ must all be codimension $\leq 1$ points of $(X \backslash \Spec A)_{\text{red}}$. As such, this finite set of points admits a common affine open neighborhood of $(X \backslash \Spec A)_{\text{red}}$ by Proposition \ref{CKincodim1}, denote it by $\Spec B \subset (X \backslash \Spec A)_{\text{red}}$. However, because $(X \backslash \Spec A)_{\text{red}}$ has the subspace topology (see \cite[Tag 04CE]{stacks}), there is a open subspace $W \subset X$ so that $W \cap (X \backslash \Spec A)_{\text{red}}=\Spec B$.

Observe that the open subspace $Y=\Spec A \cup W \subset X$ contains all the singular codimension $\leq 2$ points and satisfies the hypothesis of proposition \ref{prop:stratifyBrauer}. Therefore we have $\Br(Y)=\Br'(Y)$ and in particular $\sX|_Y$ admits a nonzero twisted vector bundle of finite rank. After taking a coherent extension of this vector bundle to all of $X$ (see \cite[15.5]{champs}) we may take a reflexive hull. Since the resulting coherent sheaf is already locally free on the singular codimension $2$ points of $\mathscr{X}$, itremains to check this at the regular codimension $2$ points. This follows by the Auslander-Buchsbaum formula. \end{proof}

\begin{remark} \label{Rpcodim2}When working with a $\mu_{pk}$-gerbe $\sX$ in characteristic $p$ note that there exists no \'etale atlas for $\sX$. Thus to check the local freeness of a reflexive module at a regular codimension $2$ point $p$ of $\sX$ one has to be careful. In particular, if $U \to \sX$ is a smooth atlas then a point lying above $p$ may not be a codimension $2$ point. That being said, there does exist a codimension $2$ point lying above $p$, thus we can apply the Auslander-Buchsbaum formula at such a point.
\end{remark}

\subsection{Non-reduced phenomena}

The statement of Corollary \ref{T1} has a generic reducedness hypothesis that we do not know how to remove. As such, we would like to address how the geometricity of a Brauer class defined over $X$ can be deduced from the geometricity of its reduction in some special cases. The first result we have is an easy consequence of the existence of the Frobenius in positive characteristic, we thank Jack Hall for bringing it to our attention. Curiously, it only concerns Brauer classes whose period is prime to the base characteristic. The second is an application of a result of de Jong regarding modifications of Azumaya algebras. 

\begin{proposition} \label{primetopthick} Let $X \to X'$ be a nilpotent thickening of algebraic spaces over $\Spec \mathbf{F}_p$. If every prime to $p$ cohomology class in $\Br'(X)$ is in $\Br(X)$ then every prime to $p$ cohomology class in $\Br'(X')$ is in $\Br(X')$. 
\end{proposition}

\begin{proof} 

Note that the ideal sheaf in $\mathcal{O}_{X'}$ defining $X$ is nilpotent, i.e. $I^m=0$ for some $m>0$. This means that if $x \in I$ then $x^{p^m}=0$. So if we denote by $F^{m}$ the iterated absolute Frobenius on $X'$, we have the following factorization 
\[\begin{tikzcd}
X \arrow[r] & X' \\
& X' \arrow[lu, "f"] \arrow[u, "F^m"]
\end{tikzcd}\]
The Frobenius pullback of a cohomology class $\alpha \in \H^2(X', \mu_n)$ is $p\alpha$, so the pullback of the cohomology class $\alpha$ under the iterated Frobenius map is $p^m\alpha$. Thus, if $\sX'$ is the gerbe over $X'$ corresponding to $\alpha$ then the gerbe $(F^m)^*\sX$ is the Brauer class corresponding to $p^m\alpha$. Moreover, by the commutative diagram above we have
\[(F^n)^*\sX' \cong f^*(\sX'|_X)\]
where $\sX'|_X$ is the gerbe corresponding to $\alpha|_X$. Since this is a geometric Brauer class, $\sX'|_X$ admits a twisted vector bundle $V$ and pulling it back via $f$ yields a vector bundle $W$. Note that if $P$ is the $\mathbf{GL}_n$-torsor associated to $V$ then $P$ is an algebraic space (see the proof of \cite[Lemma 2.12]{EHKV}). Since $f$ induces a representable morphism $\sX' \to \sX'|_{X}$, it follows that $f^*P$ is an algebraic space. It follows that the $\mu_n$-gerbe corresponding to $p^m\alpha$ is a quotient stack. However, note that there is a morphism $\sX' \to (F^m)^*\sX'$ of stacks which induces the map $(-)^{p^m}: \mu_n \to \mu_n$ on their respective bandings. Indeed, the $p^m$th power map $\mu_n \to \mu_n$ induces the map on cohomology $\H^2(X', \mu_n) \to \H^2(X', \mu_n)$ which sends $\alpha \mapsto p^m\alpha$. Since $n$ is prime to the characteristic, the $p^m$th power map induces an isomorphism of bands and thus the two stacks $\sX'$ and $(F^m)^*\sX'$ are equivalent. It follows that $\sX'$ is also a quotient stack and we may conclude using \cite[Theorem 3.6]{EHKV}. 
\end{proof}

The second result is an application of de Jong's use of elementary transformations to unobstruct Azumaya algebras on surfaces. This was described in \cite{perioddejong} via the deformation theory of Azumaya algebras and later in \cite{conickresch} using the deformation theory of sheaves on stacks.

\begin{corollary} \label{primetoptransformation} Suppose $X \to \Spec B$ is a proper, flat morphism of algebraic spaces with $B$ a noetherian complete local ring with separably closed residue field $k$ Let $X_0 =X \times_{\Spec B} \Spec k \to \Spec k$ be a projective, reduced, $2$-dimensional scheme which is equidimensional. Then every class in $\Br'(X)$ whose order is prime to the characteristic of the residue field $B/m=k$ is geometric.
\end{corollary}

\begin{proof} Fix a cohomological Brauer class $\alpha \in \Br'(X)$ of order $n$ and consider its restriction $\alpha|_{X_0} \in \Br'(X_0)$. Since $X_0$ is projective, there is an Azumaya algebra $A$ with Brauer class $\alpha|_{X_0}$ with degree $nk$ with a positive integer $k$ so that the product is prime to the base characteristic. Indeed, the main result of \cite{dejongample} shows that the restriction $\alpha|_{X_0}$ is in the geometric Brauer group. Moreover, \cite[Theorem 1]{periodben} implies there is a Azumaya algebra representing $\alpha|_{X_0}$ of degree $\text{order}(\alpha|_{X_0})k$ with $k$ prime to the characteristic. Since $\text{order}(\alpha|_{X_0})$ divides $n$ we may tensor $A$ with a suitable $\End(\mathcal{O}_{X_0}^{\oplus r})$ so that it has the desired degree without changing its Brauer class. Because $\alpha$ is annhilated by $nk$, the Kummer sequence implies there is a $\mu_{nk}$-gerbe $\sY \to X$ with the property that $[\sY] \mapsto \alpha$ under the map $\H^2(X, \mu_{nk}) \to \H^2(X, \mathbf{G}_m)$. Consider the restriction of the $\mu_{nk}$-gerbe $\sY_0=\sY \times_{X} X_0$. We know that $[A]$ and $[\sY_0]$ are equal in $\H^2(X_0, \mathbf{G}_m)$. It follows from \cite[Proposition 2.2]{perioddejong} that there is an elementary transformation $A'$ of $A$ so that $[A']=[\sY_0]$ in $\H^2(X_0, \mu_{nk})$.

Moreover, there is an elementary transformation of the Azumaya algebra $A'$ which is unobstructed and can therefore be deformed to the formal algebraic space $\widehat{X}=X \times_{B} \text{Spf} (B)$. In fact, the deformation of an Azumaya algebra $A'$ at the at the $k$th stage is obstructed by a class in $\H^2(X_0, (A'/\mathcal{O}_X)|_k) \otimes m^k/m^{k+1}$ (see \cite[Lemma 3.1]{perioddejong} and the following discussion). If $sA'$ denotes the traceless morphisms then we only need to show that $\H^0(X_0, sA' \otimes \omega_{X_0/k})=0$ as this group is dual to $\H^2(X_0, (A'/\mathcal{O})|_k)$. Apply \cite[Proposition 3.2]{perioddejong} to conclude that we may find an unobstructed Azumaya algebra $A''$ on $X_0$ whose class in $\H^2(X_0, \mu_{nk})$ is still $[\sY_0]$. It follows by the Grothendieck existence theorem (see \cite[Tag 08BE]{stacks}) that $A''$ is the restriction of an Azumaya algebra $\widehat{A''}$ on $X$. It remains to show that the Brauer class of $\widehat{A''}$ is $\alpha$ and for this it suffices to show that $\H^2(X, \mu_{nk}) \to \H^2(X_0, \mu_{nk})$ is injective. In fact, this restriction map is an isomorphism by \cite[Tag 095T]{stacks}. \end{proof}

\section{Examples in codimension $\geq 3$ arising from Artin Contractions}

The work above yields a positive answer to Grothendieck's question in codimension $2$ and so our next step will be to investigate the question in codimension $ \geq 3$. By the work of Gabber, we know that all cohomological Brauer classes on quasi-projective schemes are geometric. Thus, if we want to find interesting examples to study, one needs to find algebraic spaces which do not carry an ample line bundle. At the same time, we want these spaces to have a large cohomological Brauer group. In this section we use the work of Artin (see \cite{artincontraction}) to construct non-schematic algebraic spaces with the property that $\Br'$ is infinitely generated. \\

\subsection{Construction} Let $Y$ be a separated, finite-type scheme of dimension $\geq 2$ over an algebraically closed field $k$ containing a proper $1$-dimensional subscheme $E$ of arithmetic genus $\geq 1$. Let $I \subset \mathcal{O}_Y$ be the ideal sheaf defining $E$ and suppose that $I/I^2$ satisfies the following vanishing condition: for every coherent sheaf $F$ on $Y$ there is a $n_0$ so that 
\[\H^1(Y, S^n(I/I^2) \otimes F)=0\]
for all $n \geq n_0$ (here $S^n(-)$ denotes the $n$th symmetric power functor). Note if $I/I^2$ is a vector bundle (e.g. if $E$ and $Y$ are both smooth varieties), then the condition is equivalent to $I/I^2$ being an ample vector bundle on $E$ by \cite[Proposition 3.3]{HartshorneAmple} and the remark that follows it.

Then by Artin's constructions (see \cite{artincontraction} and \cite{MazurContraction}) there is a proper birational morphism of finite-type algebraic spaces $\pi: Y \to X$ which contracts precisely $E$ i.e. $\pi(E)$ is a point $p \in Y$ and $\pi|_{Y\backslash E}$ is an isomorphism. We begin with an elementary observation when $E$ is connected.

\subsection{Basic Properties}

\begin{lemma} \label{lem:stein} If $E \subset Y \to X$ be a contraction as above and let $\pi: Y \to X' \to X$ denote the corresponding Stein factorization. If $E$ is connected, then $\pi': X' \to X$ is a finite homeomorphism. Moreover, $X$ is a scheme if and only if $X'$ is a scheme.
\end{lemma}

\begin{proof} Since the fiber of $\pi(E)=p \in X$ is exactly the connected curve $E$ and $\pi'$ has discrete fibers, it follows that $\pi'^{-1}(p)$ cannot have more than one point in its fiber. Moreover $\pi'$ is finite by Stein factorization (see \cite[Tag 0A1B]{stacks}). Since $\pi'$ is an isomorphism away from $p$, the induced map on underlying topological spaces is a homeomorphism. Since $\pi'$ is an affine morphism, if $X$ is a scheme then $X'$ is as well. Conversely, if $X'$ is a scheme then for any point $q \in X$ there is a open affine neighborhood $q \in \Spec A=U' \subset X'$. Then the corresponding open subset $U \subset X$ is an affine open neighborhood of $q$ by \cite[Proposition 8.1]{RydhApproximation}. \end{proof}

\begin{remark} Thus if we want to construct non-schematic algebraic spaces via Artin contractions then it is enough to consider contractions where $\pi_*\mathcal{O}_Y=\mathcal{O}_X$. Moreover, in the next section we will show that if $X' \to X$ is a finite birational morphism which is an isomorphism away from finitely many points, then $\Br(X')=\Br'(X')$ implies $\Br(X)=\Br'(X)$. As such, we will assume that the contractions $Y \to X$ satisfy $\pi_*\mathcal{O}_Y=\mathcal{O}_X$ from here onwards.
\end{remark}

We now present a series of lemmas which explain how such contractions produce large cohomological Brauer groups. 

\begin{lemma} \label{contractionsequence}
If $E \subset Y$ and $\pi: Y \to X$ is a contraction with $\pi_*\mathcal{O}_Y=\mathcal{O}_X$ described above, then there is an exact sequence
\end{lemma}
\[\begin{tikzcd}
0 \arrow{r} & \Pic(X) \arrow{r} & \Pic(Y) \arrow{r} & \Pic(Y_p^{sh})  \arrow{r} & \H^2(X,\mathbf{G}_m)\\
\end{tikzcd}\]
where $p: \Spec k \to X$ is the image of $E$ and $Y_p^{sh}=Y \times_X \Spec \mathcal{O}_{X,p}^{sh}$.
\begin{proof}
Consider the exact sequence of low degree terms associated to the Leray spectral sequence for $\pi$ and $\mathbf{G}_m$:
\[0 \to \Pic(X) \to \Pic(Y) \to \H^0(X, \R^1 \pi_* \mathbf{G}_m) \to \H^2(X, \mathbf{G}_m)\]
Indeed, since the map of sheaves $\mathcal{O}_{X} \to \pi_*\mathcal{O}_Y$ is an isomorphism, we also have $\pi_*\mathbf{G}_m \cong \mathbf{G}_m$. It is well known that $\R^1\pi_* \mathbf{G}_m$ is the sheafification of the presheaf on $X_{\text{\'et}}$
\[U \mapsto \H^1(\pi^{-1}(U),\mathbf{G}_m)=\Pic(\pi^{-1}(U))\]
Moreover, because $\pi$ is an isomorphism away from $p \in X$ we find that $\R^1\pi_*\mathbf{G}_m$ is supported only at $p$ and so
\[\H^0(X,\R^1\pi_*\mathbf{G}_m)=(\R^1\pi_*\mathbf{G}_m)_p=\lim \H^1(\pi^{-1}(U), \mathbf{G}_m)=\lim \Pic(\pi^{-1}(U))\]
where the limit ranges over all \'etale neighborhoods $U$ of $p \in X$. By \cite[Lemma 3.1.16]{milneetale} or \cite[Tag 03Q9]{stacks} this group is isomorphic to
$\Pic(Y^{sh}_p)$.\end{proof}

\noindent Thus, it follows that the Picard group $\Pic(Y^{sh}_p)$ contributes to the size of $\H^2(X,\mathbf{G}_m)$. Moreover, we will see $\Pic(Y^{sh}_p)$ tends to inherit plenty of torsion from $\Pic(E)$. Next we study the map $\Pic(Y^{sh}_p) \to \Pic(E)$. Let $E_n \subset Y$ be the closed subscheme defined by the ideal sheaf $I^{n+1}$.

\begin{lemma} \label{lem:thickpicard} There exists an $n$ such that the natural map $\Pic(Y^{sh}_p) \to \Pic(E_m)$ is an isomorphism for every $m \geq n$ and there is an exact sequence of abelian groups
\[W \to \Pic(Y^{sh}_p) \cong \Pic(E_n) \to \Pic(E) \to 0.\]
If $k$ is of positive characteristic, then $W$ is torsion. Lastly, if $E$ is an elliptic curve and $I/I^2$ is a vector bundle then $\Pic(Y^{sh}_p) \to \Pic(E)$ is an isomorphism. \end{lemma}

\begin{proof} Note that we have the diagram filled with fibered squares

\[\begin{tikzcd}
E \arrow{r} \arrow{d} & \mathfrak{Y} \arrow{r} \arrow{d} & \widehat{E} \arrow{r} \arrow{d} & Y^{sh}_p \arrow{r} \arrow{d}{\pi} & Y \arrow{d} \\
\spec k(p) \arrow{r} & \text{Spf} (\mathcal{O}_{X,p}^{sh}) \arrow{r} & \spec \widehat{\mathcal{O}_{X,p}^{sh}} \arrow{r} & \spec (\mathcal{O}_{X,p}^{sh}) \arrow {r} & X\\
\end{tikzcd}
\]

We may lift line bundles from $E$ to $E_n$ for every $n$ since the obstruction to extending a given line bundle $L \in \Pic(E_n)$ to $E_{n+1}$ lies in the cohomology group
\[\H^2(E, I^n/I^{n+1})\]
on $E$ (see \cite[Tag 08VR(1)]{stacks}). However, since $E$ is a $1$-dimensional scheme, this group always vanishes. Thus every line bundle $L \in \Pic(E)$ admits a formal deformation to $\mathfrak{Y}$. Grothendieck's existence theorem as in \cite[Tag 08BE]{stacks} now yields a lift of $L$ to $\widehat{E}$. We have shown that $\Pic(\widehat{E}) \to \Pic(E_n)$ is surjective for all $n$, in fact the map is injective for sufficiently large $n$. Indeed, the kernel of the map $\Pic(E_{n+1}) \to \Pic(E_n)$ must vanish because the deformation space of a fixed line bundle on $E_n$ is given by the cohomology group
\[\H^1(E, I^n/I^{n+1})\]
by \cite[Tag 08VR(2)]{stacks}. Moreover, this group must vanish for sufficiently large $n$ because 
\[\H^1(E, S^n(I/I^2))=0\]
for sufficiently large $n$ and $S^n(I/I^2)$ surjects onto $I^n/I^{n+1}$. It follows that there is an $n$ such that $\Pic(\widehat{E}) \to \Pic(E_m)$ is an isomorphism for every $m \geq n$. Next, we show that $\Pic(Y^{sh}_p) \to \Pic(\widehat{E})$ is an isomorphism. Injectivity follows from flat base change and surjectivity follows using Artin approximation. Indeed, denote by
\[F: \text{Sch}/X \to \text{Set}\]
the functor with values $F(T)=\Pic(Y \times_X T)$. A line bundle $L$ on $\widehat{E}$ corresponds to an element $x \in F(\Spec \widehat{\mathcal{O}^{sh}_{X,p}})=F(\Spec \widehat{\mathcal{O}^{sh}_{X,p}}/m^n)$. Since $F$ is a functor locally of finite presentation (see \cite[Tag 01ZR]{stacks}), Artin's approximation argument (see \cite[Corollary 2.2]{artinapproximation}) tells us there is a \'etale neighborhood $X' \to X$ of $p$ and a point $x' \in F(X')$ so that $x=x'$ in $F(\Spec \widehat{\mathcal{O}^{sh}_{X,p}}/m^n)$. In other words there is a line bundle $L'$ on $X' \times_X Y$ restricting to $L$. However, since $X' \to X$ is an \'etale neighborhood of $p$ this means there is a line bundle $L''$ on $Y^{sh}_p$ extending $L$. 

Next, take cohomology of the following exact sequence of sheaves of groups on $E_n$
\[0 \to 1+J \to \mathbf{G}_{m,E_n} \to \mathbf{G}_{m,E} \to 0\]
where $J$ is the ideal sheaf of $E$ in $E_n$, to deduce the existence of an exact sequence
\[W=\H^1(E_n,1+J) \to \Pic(E_n) \to \Pic(E) \to 0\]
and note that $W$ is torsion if $k$ is a field of positive characteristic.

The last statement follows because if $I/I^2$ is a vector bundle, then $\text{Sym}^n(I/I^2) \cong I^{n}/I^{n+1}$ and therefore $I^n/I^{n+1}$ is ample. Indeed, the $n$th tensor power $(I/I^2)^{\otimes n}$ is ample by \cite[3.3]{BartonTensorAmple}, \cite[5.2]{HartshorneAmple} and the hypothesis that $I/I^2$ is ample. Moreover, the $n$th symmetric power $I^n/I^{n+1}$ is a quotient of the $n$th tensor power $(I/I^2)^{\otimes n}$ and ampleness descends along quotients maps by \cite[2.2]{HartshorneAmple}. It follows by \cite[1.1-1.3]{HartshorneAmpleCurves} that $\H^1(E,I^n/I^{n+1})$ vanishes for any $n \geq 1$ since $E$ is an elliptic curve.\end{proof}

\begin{remark} \label{rem:rare} Thus if $E$ is an elliptic curve and $I/I^2$ is a vector bundle, it immediately follows that $\Pic(Y^{sh}_p)$ has a torsion subgroup which is not finitely generated. On the other hand, the same is true if $E$ is an arbitrary smooth curve over a field of \emph{positive} characteristic. Moreover, if $\Br'(X)=0$ then the cokernel of the map $\Pic(Y) \to \Pic(Y^{sh}_p)$ would need to have no nonzero torsion. In particular, this would mean that all the torsion line bundles on $E$ must extend to $Y$! Next, we give examples of such contractions where $X$ is not a scheme and $\Br'(X)$ is not finitely generated. \end{remark}

\subsection{Concrete Examples} We begin with an example of surfaces which arise as contractions. 
\\\\
 $\textbf{Example}$ (Nagata's Surface) Let $k$ be an uncountable algebraically closed field. There is a normal algebraic space surface which is finite-type and proper over $k$ which is not a scheme. Fix an elliptic curve $E \subset \mathbf{P}^2_k$ with base point $p$. Choose $10$ points $p_1,...,p_{10}$ in such a way that the corresponding degree zero line bundles on $E$ have no relations. More precisely, the map
\[ \mathbf{Z}^{10} \to \langle \mathcal{O}(p_1-p),...,\mathcal{O}(p_{10}-p) \rangle \subset \Pic^{\circ}_{E/k}(k) \]
\[e_i \mapsto \mathcal{O}(p_i-p)\]
is an isomorphism. To show this is possible we will choose these points inductively. First, note that the torsion of $\Pic^{\circ}_{E/k}(k)$ is countable and since $\Pic^{\circ}_{E/k}(k)=E(k)$ is uncountable there exists a point $p_1$ corresponding to a non-torsion class. If we have chosen $p_1,...,p_{i}$ then observe that $\langle\mathcal{O}(p_1-p),...,\mathcal{O}(p_i-p)\rangle + \Pic^{\circ}_{E/k}(k)_{\mathrm{tors}}$ is still countable. Thus we may choose $p_{i+1}$ with the desired property.

Blow up $\mathbf{P}^2_k$ along $\{p_1,...,p_{10}\}$ and denote the result by $Y$ and let the strict transform of the elliptic curve be denoted by $E \subset Y$. Observe that $E^2=-1$ in $Y$ i.e. the ideal sheaf of $E$ in $Y$ restricted to $E$ has degree $-1$. This implies the conormal bundle of $E$ is ample and therefore we may contract $E$ in $Y$. Let $f: Y \to X$ denote the contraction with the property that $f_*\mathcal{O}_Y=\mathcal{O}_X$. Normality follows because the map $\mathcal{O}_Y \to f_*\mathcal{O}_X$ is an isomorphism. It remains to check that $X$ cannot be a scheme.

If it were, there would exist an affine open neighborhood about the image of the elliptic curve $f(E)=q$, call it $\Spec A \subset X$. First note that the complement of a dense affine open neighborhood in a separated normal scheme has pure codimension $1$ (see \cite[Tag 0BCV]{stacks}). Thus there is a curve, $C$, in $X$ not meeting $q$. It follows that $C$ (viewed as a divisor in $Y$) cannot meet $E$. Moreover, in $\mathbf{P}^2_k$ the closure of $C$ can only meet $E$ at $\{p_1,...,p_{10}\}$. But since $C$ is linearly equivalent to $nL$ for a line $L \subset \mathbf{P}^2$ it follows that the $p_i$ have a nontrivial relation in $\Pic^{\circ}_{E/k}(k)$, a contradiction.
To see that $\Br'(X)$ is not finitely generated, observe that $\Pic(Y)$ is a finitely generated abelian group whereas the torsion subgroup of $\Pic(Y^{sh}_p)=\Pic(E)$ is not and apply Lemma \ref{lem:thickpicard} and \ref{contractionsequence}.
\\\\
Next, we construct similar examples without the requirement that the base field be uncountable and where $\dim X \geq 3$. The situation here is slightly more delicate because there isn't a degree-theoretic criterion for ampleness for vector bundles of higher rank. As such, we use a degeneration argument to find a smooth hypersurface $Y \subset \mathbf{P}^n$ containing a fixed curve $E$ with $N_{E/Y}^{\vee}$ ample. Moreover, we would like the resulting contraction $Y \to X$ to have a large cohomological Brauer group and so we want $\Pic(Y^{sh}_p)$ to have plenty of torsion. In positive characteristic even if $\Pic(Y^{sh}_p) \to \Pic(E) \to 0$ is not an isomorphism, the inverse image of the torsion subgroup is still torsion and hence infinitely-generated. In characteristic zero, we may arrange for this to be an isomorphism by using the fact that the deformation spaces $\H^1(E,S^n(N_{E/Y}^{\vee}))$ vanish for every $n \geq 1$.
\begin{lemma} \label{lem:semi-stable} Let $E$ be a smooth proper connected curve of genus $g$ over a field $k$ of characteristic $0$. If a semi-stable vector bundle $V$ has $\mu(V)>2g-2$ then $\H^1(E,S^n(V))=0$ for every $n \geq 1$. 
\end{lemma}
\begin{proof} Note that if a semi-stable vector bundle $W$ has $\mu(W)>\mu(\omega_{E/k})=2g-2$ then by Serre duality we have
\[\H^1(E,W)^{\vee}=\Hom_{\mathcal{O}_E}(W,\omega_{E/k})=0\]
Indeed, the image of such a nontrivial map would have slope smaller than that of $\mu(W)$, violating the semistability assumption. Thus, since $S^n(V)$ is semi-stable for every $n \geq 1$ by \cite[Theorem 0.2]{gieseker}, we only need to verify that $\mu(S^n(V))>2g-2$. But this follows because $\mu(S^n(V))=n\mu(V)>2g-2$ for every $n\geq 1$. \end{proof}

\begin{proposition} \label{examples} Let $k$ be an algebraically closed field and $n \geq 3$. There are $n$-dimensional non-schematic algebraic spaces $X$ which are proper over $k$ which are equipped with a proper birational morphism $f: Y \to X$ such that 
\begin{enumerate} 
\item $f$ is an isomorphism away from finitely many points, 
\item $Y$ quasi-projective over $k$, 
\item the exceptional locus of $f$ is $1$-dimensional, and
\item the cohomological Brauer group of $X$ is not finitely-generated.
\end{enumerate}
\end{proposition}

\begin{proof}

 Let $k$ be an algebraically closed field and consider a complete and connected nonsingular curve $E \subset \mathbf{P}^n_k$ of genus $g \geq 1$ in projective $n$-space ($n \geq 4$) which is a complete intersection cut out by polynomials $g_1,...,g_{r-1}$ of degrees $d_1,...,d_{r-1}$. We will show that there is a smooth hypersurface $Y$ containing $E$ with the property that the conormal bundle $N_{E/Y}^{\vee}$ is an ample vector bundle. In characteristic $0$, we will also show that $Y$ can be chosen so that if $I$ is the ideal sheaf of $E$ in $Y$ then $\H^1(E, S^n(I/I^2))=0$ for every $n$. Then, the maps
\[\Pic(Y^{sh}_p) \to \Pic(E_n) \to \Pic(E)\]
are isomorphisms for every $n$, thereby showing that the torsion subgroup of $\Pic(Y^{sh}_p)$ is infinitely-generated. As we have noted above, this is automatic in positive characteristic. In either case, we will use this to show that the resulting Artin contraction $Y \to X$ has an infinitely-generated cohomological Brauer group.

Choose $m>2g-2 \geq 0$ and observe that because $E$ is a curve, we may find a surjection 
\[\bigoplus_{i=1}^{r-1} \mathcal{O}(-d_i)|_E \to \mathcal{O}(m)|_E^{\oplus r-2} \to 0\]
Thus, by dualizing this surjection we obtain an exact sequence
\[0 \to \mathcal{O}(-m)|_E^{\oplus r-2} \to \bigoplus_{i=1}^{r-1} \mathcal{O}(d_i)|_E \to \mathcal{O}(c)|_E \to 0\]
where $c>0$. Observe that the natural restriction map
\[\Hom_{\mathcal{O}_{\mathbf{P}^n}}(\bigoplus_{i=1}^{r-1} \mathcal{O}(d_i), \mathcal{O}(c)) \to \Hom_{\mathcal{O}_E}(\bigoplus_{i=1}^{r-1} \mathcal{O}(d_i)|_E, \mathcal{O}(c)|_E)\]
is surjective because $E \subset \mathbf{P}^n$ is a complete intersection. Thus, we may choose $r-1$ homogenous polynomials $a_1,...,a_{r-1}$ of degrees $c-d_1,...,c-d_{r-1}$ in $k[x_0,...,x_n]$ such that the corresponding map in $\Hom_{\mathcal{O}_{\mathbf{P}^n}}(\bigoplus_{i=1}^{r-1} \mathcal{O}(d_i), \mathcal{O}(c))$ gives rise to the exact sequence of vector bundles on $E$ as above. Then by \cite[Corollary 6.16]{3264}, the polynomial $g=\Sigma_{i=1}^{r-1}a_ig_i$ defines a degree $c$ hypersurface $Y$ containing $E$ with the property that $N_{E/Y}^{\vee}=\mathcal{O}(m)|_E^{\oplus r-2}$. In particular, the conormal bundle is ample and semi-stable. 

The problem is that $Y$, as defined, may not be smooth. However, if we perturb the choice of $a_1,...,a_{r-1}$, or, equivalently, perturb the map in the affine space 
\[\Hom_{\mathcal{O}_{\mathbf{P}^n}}(\bigoplus_{i=1}^{r-1} \mathcal{O}(d_i), \mathcal{O}(c))\] 
then we can conclude that the new $Y$ is smooth with the vector bundle $N_{E/Y}^{\vee}$ ample and semi-stable. Indeed, ampleness is an open condition by \cite[Proposition 4.4]{HartshorneAmple} and smoothness of the general member of a linear series with a $1$-dimensional base locus is also nonempty and open by \cite[Proposition 5.6]{3264}. Finally, openness of semistability follows from \cite[Proposition 2.3.1]{lehn}. Note that even after the perturbation, $N_{E/Y}^{\vee}$ still has a slope larger than $\mu(\omega_{E/k})=2g-2$, and hence if $\text{char}(k)=0$ we have $\H^1(E,S^n(N_{E/Y}^{\vee}))=0$ for every $n\geq 1$ by Lemma \ref{lem:semi-stable}.

Next, we will show that the resulting Artin contraction produces a non-schematic algebraic space. Observe that the natural map
\[\mathbf{Z} \cong \Pic(\mathbf{P}^n_k) \to \Pic(Y)\]
is an isomorphism by the Grothendieck-Lefschetz theorem (see \cite[Corollary 3.7]{SGA2}). Thus every nontrivial effective divisor on $Y$ is ample. Moreover, there exists a separated algebraic space $X$ along with a proper and birational morphism $\pi: Y \to X$ which contracts $E$. We claim that $X$ is not a scheme. Suppose it were, then there would exist an affine neighborhood $\Spec A \subset X$ about the singular point. Since this neighborhood is dense the complement $X \backslash \Spec A$ has pure codimension $1$ by \cite[Tag 0BCV]{stacks}. It follows that there is a codimension $1$ closed subvariety $X\backslash \Spec A = D \subset Y$ which does not pass through $E$. Consider the line bundle
\[\mathcal{O}(D) \in \Pic(Y) \cong \mathbf{Z}\]
and note that it has a section which vanishes along a non-empty locus. It follows that $\mathcal{O}(D)$ must be ample on $Y$, but this contradicts the fact that it doesn't pass through $E$. Thus there is no affine neighborhood of the singular point in $X$, i.e. $X$ is a non-schematic algebraic space.
To see that the cohomological Brauer group of $X$ is not finitely generated, note that the Leray spectral sequence with $Y \to X$ and $\mathbf{G}_m$ yields an exact sequence
\[\mathbf{Z} \to \Pic(Y^{sh}_p) \to \H^2(X, \mathbf{G}_m)\]
Moreover, the torsion subgroup of $\Pic(Y^{sh}_p)$ is not finitely generated. Thus, we obtain examples in all dimensions of non-schematic algebraic spaces with infinitely generated cohomological Brauer group. \end{proof}

\subsection{Progress over finite fields in dimension 3}

We have shown that Grothendieck's question has an affirmative answer for reduced, noetherian, separated algebraic spaces away from a subset of codimension $\geq 3$. The next natural question is: does it hold for threefolds? If $X$ is a scheme, any point admits an affine neighborhood and hence Gabber's theorem tells us every cohomology class is Zariski-locally represented by an Azumaya algebra. Even for threefolds which are not schemes, it is not clear that such a result should hold \emph{near} a fixed closed point. Here we give an affirmative answer when $X$ is a $3$-dimensional separated algebraic space, finite-type over a finite field $k$. An essential ingredient is the following result of Artin: $2$-dimensional algebraic space surfaces which are normal and defined over the algebraic closure of a finite field are always quasiprojective.

\begin{Theorem} \label{dim3finite} Let $X$ be a $3$-dimensional separated algebraic space, finite-type over $\overline{\mathbf{F}}_p$. Then there is a Zariski open covering $\bigcup U_i=X$ with the property that $\Br(U_i)=\Br'(U_i)$. Moreover, for any finite set of points $x_1,...,x_n \in X$ we can find a open set $U \subset X$ containing the $x_i$ so that $\Br(U)=\Br'(U)$.
\end{Theorem}

\begin{proof} The first claim follows from the second, so we prove that instead. Fix a dense open affine neighborhood $\Spec A \subset X$ and consider the complement $X \backslash U=S$. This $S$ is a $2$-dimensional algebraic space which is finite-type and separated over $\overline{\mathbf{F}_p}$. Consider the normalization $S' \to S$ and note that because $S$ is excellent this morphism is finite. We may compactify $S'$ (see \cite[1.2.1]{conradnagata}) and then by applying \cite[4.6]{artinalgebraicspaces} it follows that $S'$ is a projective scheme. Then by \cite[2.3]{GrossRes} or \cite[Corollary 48]{finitekollar}, it follows that $S$ is a scheme with the Chevalley-Kleiman property. In particular, given any finite set of points $x_1,...,x_n \in X$, there is an open affine subscheme $\Spec B \subset S$ containing the $x_i$ which are not in $\Spec A$. Since $S \subset X$ has the subspace topology there is a open subset $W \subset X$ with the property that $W \times_X S=\Spec B$. 
To conclude, apply Corollary \ref{prop:stratifyBrauer} to $W \cup \Spec A$ to see that 
\[\Br(W \cup \Spec A) = \Br'(W \cup \Spec A)\]
as desired. \end{proof}

\section{Brauer classes on Contractions are Geometric}

Fix a base field $k$ and let $\bar{k}$ denote its algebraic closure.

\subsection{A general collection of sections is a frame away from a Cartier divisor}

The proof of the following lemma is explained within the proof of \cite[Theorem 3.1]{enoughAzumaya} in the $1$-dimensional setting. That argument essentially follows an idea of Maruyama (see \cite[Theorem 2.3]{Maruyama}).

\begin{lemma} \label{trivialelementary} Let $X$ denote a scheme of dimension $\leq 3$ which is proper over $\bar{k}$. Suppose $E$ is a globally generated rank $r$ vector bundle on $X$. Then, given any finite set of points $Z \subset X$, there is a Cartier divisor $i: D \to X$ not meeting $Z$, a line bundle $L$ on $D$ and an exact sequence
\[0 \to \mathcal{O}_X^{\oplus r} \to E \to i_*L \to 0\]
\end{lemma}

\begin{proof} Consider the finite-dimensional vector space $V=\H^0(X,E)$ and the Grassmanian of $r$-dimensional quotients of $V$ over $\bar{k}$, we denote it by $\mathbf{Gr(V, r)}$. Note that the canonical surjection $V \otimes \mathcal{O}_X \to E$ defines a morphism $f: X \to \mathbf{Gr(V, r)}$ by pulling back the universal quotient $V \otimes \mathcal{O}_{\mathbf{Gr(V, r)}} \to E_U$. Fix some $r$-dimensional subvector space of $i: V' \subset V$ and consider the subscheme $G_{j} \subset \mathbf{Gr(V, r)}$ of $r$-dimensional quotient maps $\phi: V \to V''$ where the composition $\phi \circ i: V' \to V''$ has rank $\leq j$. The idea will be to translate $V'$ so that the map $f: X \to \mathbf{Gr(V, r)}$ intersects $G_j$ transversely and so that the image of the associated points of $X$ avoid $G_{r-1}$.

Note that $G_j$ is an integral subscheme of codimension $(r-j)^2$. To see this, first observe that the open subset of rank $r$ matrices in $\mathbf{F} \subset \mathbf{Hom}_k(V, k^{\oplus r}) \cong \mathbf{M}_{n\times r}(k)$ admits a natural map to $\mathbf{Gr}(V, r)$:
\[\phi: V \to k^{\oplus r} \to 0 \mapsto [\phi: V \to k^{\oplus r} \to 0]\]
Note that this gives $\mathbf{F}$ the structure of a principle $\mathbf{GL}_r$-bundle over $\mathbf{Gr(V,r)}$. Consider $V' \otimes \mathcal{O}_{\mathbf{F}}$ and the universal map
\[V' \otimes_k \mathcal{O}_{\mathbf{F}} \to V \otimes_k \mathcal{O}_{\mathbf{F}} \to \mathcal{O}_{\mathbf{F}}^{\oplus r}\]
Observe that $G_j \times_{\mathbf{Gr(V, r)}} \mathbf{F}$ is the locus where this composition has rank $\leq j$. It is integral of codimension $(r-j)^2$ by \cite[1.1, 2.10]{det} and therefore so is  $G_j$ since this property can be checked smooth locally.

It follows that $G_{r-1}$ is an integral Cartier divisor and $G_{r-2}$ has codimension $4$ in $\Gr(V, r)$. Note that $\mathbf{GL(V)}=G$ acts on $\mathbf{Gr(V, r)}$ transitively. Let $\bigcup_i X_i \to X$ be the disjoint union of the components of the normalization of $X$, and by precomposition, for each $i$ and $j$ we obtain morphisms $X_i \to X \to \mathbf{Gr(V, r)}$ and $G_j \to \mathbf{Gr(V,r)}$. By applying \cite[Theorem 2(1)]{translate} we see that there is a nonempty open set $U_{ij} \subset G$ so that for each rational point $s \in U_{ij}$ the fiber product with the translate $sG_j \times_{\mathbf{Gr(V,r)}} X_i$ has the expected dimension:
\[\dim(sG_j)+\dim(X_i)-\dim(\mathbf{Gr(V,r)})\]
Note that $sG_j$ is the locus of quotients $V \to k^{\oplus r}$ where the composition $s^{-1}V' \subset V \to k^{\oplus r}$ has rank $\leq j$. Thus, by choosing a rational point of the intersection $\bigcap_{i,j} U_{ij}$ we obtain a translate $s^{-1}V'$ so that the associated subscheme $G_{r-1}$ pulls back to some Cartier divisor of $X_i$ but $X_i \times_{\mathbf{Gr(V, r)}} G_{r-2}$ is empty for each $i$ (this is where the dimension of $X$ is used). By a similar argument, one can further show that there is a translate $s^{-1}V'$ with all the properties above so that $sG_{r-1}$ does not contain the image of any of the associated points of $X$ or any point from $Z$. Indeed, these points correspond to (finitely many) irreducible subschemes $C_k \subset X$ and so for each there is a open subset $U_k \subset G$ so that for every rational point $s \in U_k$ the translate $sG_{r-1}$ does not contain the image of $C_k$ in $\mathbf{Gr(V,r)}$. Taking a rational point from the intersection of all the $U_{ij}$ and $U_k$ we obtain a translate $s^{-1}V'$ where $D=sG_{r-1} \times_{\mathbf{Gr(V,r)}} X$ is a Cartier divisor containing no associated points of $X$ or any point from $Z$ and for which $sG_{r-2} \times_{\mathbf{Gr(V,r)}} X$ is empty. Thus, we replace $V'$ with $s^{-1}V'$.

Consider the following exact sequence
\[ 0 \to V' \otimes \mathcal{O}_{\mathbf{Gr(V, r)}} \to E_U \to L_{G_{r-1}} \to 0\] 
where the last term is the cokernel of the first map. Exactness on the left follows because the map is an isomorphism away from the Cartier divisor $G_{r-1} \subset \mathbf{Gr(V, r)}$. The cokernel is supported on the reduced subscheme $G_{r-1}$ and away from $G_{r-2}$ the cokernel has constant rank $1$ i.e. away from $G_{r-2}$ the last term is a line bundle.

Therefore by pulling back this exact sequence along $X \to \mathbf{Gr(V,r)}$ we obtain an exact sequence
\[0 \to \mathcal{O}_X^{\oplus r} \to E \to i_*L \to 0\]
for some line bundle $L$ on $D$. The only thing to verify is injectivity on the left but this follows because the obstruction to injectivity is a coherent sheaf supported on the Cartier divisor $G_{r-1} \times_{\mathbf{Gr(V,r)}} X=D \subset X$. By construction this divisor contains no associated points of $X$ and thus no sheaf supported on it admits a nontrivial map to $\mathcal{O}_X^{\oplus r}$.
\end{proof}

\subsection{Bertini-type results}

\begin{lemma} \label{E-ample} Let $C \subset Y$ be the inclusion of a proper $1$-dimensional closed subscheme in a quasiprojective scheme $Y$ over $\bar{k}$. Suppose that $E$ is rank $r$ vector bundle on $C$ whose determinant extends to $Y$, then there is an ample line bundle $L$ on $Y$ with the property that $E \otimes L|_C$ is globally generated and if $\mathrm{det}(E \otimes L|_C) =\mathcal{O}_C(D)$ for some effective Cartier divisor $D \subset C$ not containing any associated point of $Y$, then there is an ample Cartier divisor $H \subset Y$ with $H \cap C=D$ scheme-theoretically. \end{lemma}

\begin{proof} We may embed $Y$ into a projective scheme $\bar{Y}$ over $\bar{k}$. We may also find a coherent extension, $G$, of the determinant of $E$ to $\bar{Y}$ which is a line bundle on $Y \subset \bar{Y}$. By flatification (see \cite[Tag 0815]{stacks}) there is a $Y$-admissible blow-up $\bar{Y}' \to \bar{Y}$ so that $G|_{Y}$ extends to a line bundle on $\bar{Y}'$. Thus, we may replace $\bar{Y}$ with $\bar{Y}'$ so that the determinant of $E$ on $C \subset \bar{Y}$ extends to a line bundle $L'$ on $\bar{Y}$. In particular, we may assume that $Y$ is projective.

Fix closed points $p_1,...,p_l \in Y \backslash C$, one from each embedded point of $Y$ that is not contained in $C$, and isomorphisms $f^r_{p_i}: L' \otimes L^{\otimes r} \otimes k(p_i) \simeq k(p_i)$. Using these isomorphisms we may uniquely describe the value of a section $s \in \H^0(Y, L' \otimes L^{\otimes r})$ at $p_i$ as an element of $k(p_i)$. Now, choose an ample line bundle $L$ on $Y$ with the following properties
\begin{enumerate}
\item $L' \otimes L^{\otimes r}$ is ample, 
\item $E \otimes L|_C$ globally generated, 
\item $\H^1(Y, L' \otimes L^{\otimes r} \otimes I_C)=0$, and
\item There exists $s_1,...,s_l \in \H^0(Y,L' \otimes L^{\otimes r})$ with $s_i(p_j)=\delta_{ij}$ and $s_i|_C=0$.
\end{enumerate} 
Indeed, these items can be arranged by taking $L$ to be sufficiently ample. The first three items are well-known, for the last apply the avoidance lemma as stated in \cite[Theorem 5.1]{gabber2015hypersurfaces}. After possibly replacing $L$ with $L^{\otimes N}$, the aformentioned result allows us to find sections $s_i' \in \H^0(Y,L' \otimes L^{\otimes r})$ which vanish on $C \cup \{p_1,..,p_{i-1},p_{i+1},..,p_l\}$ but is nonzero at $p_i$. After finding an $N$ so that such $s_1',..,s_l'$ exist in $\H^0(Y,L' \otimes L^{\otimes rN})$, we use isomorphisms $f^{rN}_{p_i}$ and set $s_i=\frac{1}{s_i'(p_i)}s_i'$ to obtain the desired $s_i$. Thus, by replacing $L$ with $L^{\otimes N}$ we may achieve the four properties above.

We claim that this $L$ satisfies the statement of the lemma. Indeed, if $\det(E \otimes L|_C)=(L' \otimes L^{\otimes r})|_C=\mathcal{O}_C(D)$ for some effective divisor $D \subset C$ not containing any associated point of $Y$, then taking cohomology of the exact sequence
\[0 \to L' \otimes L^{\otimes r} \otimes I_C \to L' \otimes L^{\otimes r} \to (L' \otimes L^{\otimes r})|_C \to 0\]
yields a surjection $\H^0(\bar{Y}, L' \otimes L^{\otimes r}) \to \H^0(C,(L' \otimes L^{\otimes r})|_C)$. Thus, we may consider the preimage of the section $s_D \in \H^0(C, \mathcal{O}_C(D))$ to obtain a section $s_{H'}$ whose vanishing $H'$ intersects $C$ precisely at $D$. 

However, $H'$ may not be Cartier since it could contain associated points of $Y$. On $C$, $H'$ contains no associated points because $H' \cap C=D$ and $D$ doesn't contain any  associated points of $Y$. However, the vanishing of the section
\[s=s_{H'}+\Sigma_{i=1}^l (1-s_{H'}(p_i))s_i \in \H^0(Y, L' \otimes L^{\otimes r})\]
does not contain any associated point of $Y$ and is therefore a Cartier divisor $H$ with $H \cap C=D$ scheme-theoretically.
\end{proof}

The next lemma is a twisted analog of the following fact: given two vector bundles whose ranks differ by a sufficiently large integer on a quasi-projective scheme $H$ and a surjection between them defined on a finite subscheme $D$, the surjection can be extended to a surjection defined over all of $H$. 

\begin{lemma} \label{Bertini} Let $H$ denote a quasi-projective scheme over $\bar{k}$ and suppose that $D \subset H$ is a finite subscheme. Fix a $\mu_n$-gerbe $\ms H \to H$, let $\ms D \to D$ denote its restriction to $D$, and let $\chi$ denote the twisted line bundle on $\ms D$ (it is unique upto non-unique isomorphism). Suppose that $W_1$ and $W_2$ are twisted vector bundles on $\ms H$ of rank $rl$ and $l$ respectively with $l(r-1)+1> \dim H$. Given a surjective map $\phi_D: W_1|_{\ms D} \to W_2|_{\ms D}$, then there is an ample line bundle $L \in \Pic(H)$, a $N>0$, fixed isomorphisms $W_2|_{\ms D} \otimes L^{\otimes N} \cong \chi^{\oplus l} \cong W_2|_{\ms D}$, and a surjection $\psi: W_1 \to W_2 \otimes L^{\otimes N}$ with the property that there is an isomorphism $g: \chi^{\oplus l} \to \chi^{\oplus l}$ making the following diagram commute

\centering
\begin{tikzcd}
W_1|_{\ms D} \arrow{d}{\mathrm{id}} \arrow{r}{\psi|_D} & W_2|_{\ms D} \otimes L^{\otimes N} \cong \chi^{\oplus l} \arrow{d}{g} \\
W_1|_{\ms D} \arrow{r}{\phi_D} & W_2|_{\ms D} \cong \chi^{\oplus l}  \\
\end{tikzcd} \end{lemma}
\begin{proof} Set $W=\mathcal{H}om(W_1,W_2)$, since it is an untwisted vector bundle on $\ms H$ we may consider it as a vector bundle of rank $rl^2$ on $H$. Since $H$ is a quasi-projective scheme we may embed $H \subset H'$ into a projective $k$-scheme $H'$. We may also find a coherent extension of $W$ to $H'$. Moreover $W$ can be made locally free after a further $H$-admissible blow up $\bar{H} \to H'$ by flatification (see \cite[Tag 0815]{stacks}). The result is a projective compactification $\bar{H}$ of $H$ where $W$ extends as a locally free module, by abuse of notation we refer to this vector bundle on $\bar{H}$ as $W$. Fix an ample line bundle $L$ on $\bar{H}$.

Since $D$ is a finite subscheme we may write $D=\bigcup_{s \in D} \Spec A_s$ where the latter is a disjoint union of Artinian local rings. Choose identifications $L|_D \cong \mathcal{O}_D$ and $W_2|_{\ms D} \cong \chi^{\oplus l}$ once and for all and note that, for each $N>0$, a fixed choice of such isomorphisms induces a canonical isomorphism 
\[W_2|_{\ms D} \otimes L|_D^{\otimes N} \cong \chi^{\oplus l}\]
After fixing these identifications there is a unique way of identifying maps $t: W_1|_{\ms D} \to W_2|_{\ms D} \otimes L^{\otimes N}$ and $s:W_1|_{\ms D} \to W_2|_{\ms D}$ with maps in $\Hom(W_1|_{\ms D}, \chi^{\oplus l})$. Since we have made a fixed choice, we will continue to make these fixed identifications without further mention.

We will show that for some $N>0$ there is a section $\psi \in W \otimes L^{\otimes N}$ so that $\psi|_H:W_1 \to W_2 \otimes L^{\otimes N}$ is surjective. Moreover, this $\psi$ will have the property that for each $s \in D$ there is a $c_s \in k^*$ where $(\phi_D)|_{\Spec A_s}=c_s\psi|_{\Spec A_s}$. Then, multiplication by $(c_s)_{s \in D}=g \in \Pi_{s \in D} A_s$ yields a commutative diagram

\begin{center}
\begin{tikzcd}
W_1|_{\ms D} \arrow{d}{\text{id}} \arrow{r}{\psi|_D} & W_2|_{\ms D} \otimes L^{\otimes N} \cong \chi^{\oplus l} \arrow{d}{g} \\
W_1|_{\ms D} \arrow{r}{\phi_D} & W_2|_{\ms D} \cong \chi^{\oplus l}  \\
\end{tikzcd}
\end{center}
Consider the sheaf $W \otimes L^{\otimes N}$ and the associated exact sequence
\[0 \to W \otimes L^{\otimes N} \otimes I_D \to W \otimes L^{\otimes N} \to  (W \otimes L^{\otimes N})|_D \to 0\]
If $N$ is large enough we may assume that $W \otimes L^{\otimes N} \otimes I_D$ is globally generated. It follows that given any point $p \in H$ which is not in $D$, there are $rl^2$ global sections of $W \otimes L^{\otimes N}$ which vanish along $D$ but form a basis for $W \otimes L^{\otimes N} \otimes k(p)$. Moreover, for each $s \in D$ we have an exact sequence
\[0 \to (W \otimes L^{\otimes N} \otimes I_{D\backslash \{s\}}) \otimes I_{\Spec A_s} \to W \otimes L^{\otimes N} \otimes I_{D\backslash \{s\}} \to (W \otimes L^{\otimes N})|_{\Spec A_s}\cong W|_{\Spec A_s} \to 0\]
given by tensoring the exact sequence defining $\Spec A_s \subset \bar{H}$ with $W \otimes L^{\otimes N} \otimes I_{D\backslash \{s\}}$. Exactness on the right follows because $I_{D \backslash \{s\}}$ is isomorphic to $\mathcal{O}_{\bar{H}}$ near the closed subscheme $\Spec A_s$. Now, if $N$ is large enough, this yields an exact sequence of $k$-vector spaces upon taking global sections. Thus, if we take $N$ large enough, for every $s \in D$ there is a global section 
\[\psi_s \in \H^0(\bar{H}, W \otimes L^{\otimes N} \otimes I_{D\backslash \{s\}})\]
lifting $(\phi_D)|_{\Spec A_s}$ but which vanishes at every other $s' \in D$. Moreover we may also take a $k$-basis $\psi_1,...,\psi_n$ for the $k$-vector space $\H^0(\bar{H},W \otimes L^{\otimes N} \otimes I_D)$. Viewing all these sections inside $\H^0(\bar{H}, W \otimes L^{\otimes N})$ we set $\mathbf{A}^M_k=\Spec k[x_1,...,x_n,x_s;s \in D]$ and observe that there is a universal section
\[\psi=\Sigma_{i=1}^n x_i\psi_i+\Sigma_{s \in D} x_s\psi_s\]
of $W \otimes L^{\otimes N}$ pulled back to $\mathbf{A}^M_k \times_k \bar{H}$. Thus, for each rational point $a=(a_1,..,a_n,a_s;s \in D) \in \mathbf{A}^M_k$ we obtain the section 
\[\psi_a=\Sigma_{i=1}^n a_i\psi_i+\Sigma_{s \in D} a_s\psi_s \in \H^0(\bar{H},W \otimes L^{\otimes N})\]
and over the open set $H$ the restriction of this section $\psi_a|_H$ corresponds to a morphism in $\Hom(W_1,W_2 \otimes L^{\otimes N})$.

Over the complement $U=H \backslash D \subset H$ we consider the closed locus of non-surjective maps
\[Z=\{(a,u)| \psi_a \text{ is not surjective}\} \subset \mathbf{A}^M_k \times U\]
For any $u \in U$ the fiber $Z_u$ has codimension $rl-l+1$ since the $\psi_i$ generate $W \otimes L^{\otimes N}$ at all $u \in U$. Indeed, there is a surjective linear map
\[\pi: \mathbf{A}^M_{k(u)} \to \Hom_{k(u)}(k(u)^{rl},k(u)^{l})=(W \otimes L^{\otimes N}) \otimes k(u)\]
and the (irreducible) determinantal variety $B$ consisting of rank $\leq l-1$ linear maps has codimension $rl-l+1$ (see, for instance, \cite[1.1, 2.10]{det}) and therefore the preimage $\pi^{-1}(B)=Z_u$ has the same codimension. Thus, the dimension of $Z$, and hence the closure of its image $p_1(\bar{Z})$ in $\mathbf{A}^M_k$, is at most 
\[M-(rl-l+1)+\dim(H)<M\]
Moreover, consider $Z_s=V(x_s) \subset \mathbf{A}^M_k$ and note that this is a divisor. Thus there must necessarily be a $k$-point $c \in \mathbf{A}^M_k$ avoiding $\bar{Z} \cup \bigcup_{s \in D} Z_s$, we claim that the corresponding section of $W \otimes L^{\otimes N}$ works as desired. Indeed, such a point corresponds to a section
\[\psi_c=\Sigma_{i=1}^n c_i\psi_i+\Sigma_{s \in D} c_s\psi_s \in \H^0(\bar{H},W \otimes L^{\otimes N})\]
which is a surjective linear map for every $u \in U$ because $c$ is not in $\bar{Z}$. Moreover, $c_s \in k^*$ for each $s \in D$ because $c$ is not in $V(x_s) \subset \mathbf{A}^M_k$. Finally, for each $s \in D$ we have
\[\psi_c|_{\Spec A_s}=c_s\psi_s|_{\Spec A_s}=c_s\phi_D|_{\Spec A_s}\]
since $\psi_{s'} \in \H^0(\bar{H},W \otimes L^{\otimes N} \otimes I_{D\backslash \{s'\}})$ for all $s' \in D\backslash \{s\}$ and $\psi_i \in \H^0(\bar{H},W \otimes L^{\otimes N} \otimes I_D)$ for all $1 \leq i \leq n$.
\end{proof}

\subsection{Descending Twisted Vector Bundles}

Let $f: Y \to X$ be a proper birational morphism of algebraic spaces which are both finite-type over an algebraically closed field $k$ with $f_*\mathcal{O}_Y=\mathcal{O}_X$. Suppose that $f$ is an isomorphism away from finitely many points $Z=\{p_1,..,p_j\} \subset X$. From here onwards, the following notation is in force
\begin{enumerate}
\item $U=X\backslash Z$ is the locus where $f$ is an isomorphism.
\item $Z^{sh}$ is the disjoint union of the strictly henselian local rings of each $p_i$ in $X$ and $E^{sh}=Z^{sh} \times_X Y$.
\item $\widehat{Z}$ is the disjoint union of the completions of each of the strictly local rings of each $p_i$ in $X$ and $\widehat{E}=\widehat{Z} \times_X Y$.
\end{enumerate}
In summary, we have a commutative diagram with cartesian squares:
\[\begin{tikzcd}
E \arrow{r} \arrow{d}  & \widehat{E} \arrow{r} \arrow{d} & E^{sh} \arrow{r} \arrow{d} & Y \arrow{d}{f} & U \arrow{l} \arrow{d}{\text{id}} \\
Z  \arrow{r}  & \widehat{Z}=\bigsqcup_{i=1}^j \spec \widehat{\mathcal{O}_{X,p_i}^{sh}} \arrow{r} & Z^{sh}=\bigsqcup_{i=1}^j \spec (\mathcal{O}_{X,p_i}^{sh}) \arrow {r} & X & U \arrow{l}\\
\end{tikzcd}
\]

Our goal is to descend an Azumaya algebra from $Y$ to $X$ and because $f_*\mathcal{O}_Y=\mathcal{O}_X$ to descend a vector bundle it suffices to show that the bundle is trivial on $\widehat{E}$. In fact, in \cite{PerlingToric} Schr\"oer-Perling showed that the triviality of a vector bundle on $\widehat{E}$ can be checked on a suitable $\emph{infinitesimal}$ neighborhood of $E$. In turn, this will allow us to give a sufficient criterion for descending \emph{twisted} vector bundles from $Y$ to $X$. Moreover, this will be enough to descend the corresponding Azumaya algebra since any Azumaya algebra is of the form $\End(V)$ for a twisted vector bundle $V$. To fix the cohomology class of the twisted vector bundle we fix a $\mu_n$-gerbe $\ms X$ over $X$, and if $t: T \to X$ is a morphism then we use the notation $\ms X_T=\ms X \times_X T$ to denote the pullback of $\ms X$ along $t$. Observe that because $\widehat{Z}$ is the disjoint union of strictly local rings $\ms X_{\widehat{Z}}$ admits a twisted line bundle $\chi$ which is unique upto (non-unique) isomorphism. 

\begin{lemma} \label{twisteddescent} Let $f: Y \to X$, $g: \ms X_Y \to \ms X$ and $\chi \in \Pic(\ms X_{\widehat{Z}})$ be as above. If a twisted vector bundle $V$ on $\ms X_Y$ becomes isomorphic to $\chi^{\oplus k}|_{\ms X_{\widehat{E}}}$ when restricted to $\ms X_{\widehat{E}}$ then $g_*V$ is a twisted vector bundle on $\ms X$. 
\end{lemma}

\begin{proof} Since $f$ is an isomorphism away from $Z$ it suffices to analyze $g_*V$ over $\widehat{Z}$. Note that by flat base change we have $g_*\mathcal{O}_{\ms X_{\widehat{E}}}=\mathcal{O}_{\ms X_{\widehat{Z}}}$ so we may assume $\widehat{Z}=X$. Now use the projection formula $g_*V=g_*g^*\chi^{\oplus k}=g_*\mathcal{O}_{\ms X_{\widehat{E}}} \otimes \chi^{\oplus k}=\chi^{\oplus k}$. The result follows.\end{proof}

\begin{lemma} \label{infinitesimal} Let $f: Y \to X$, $\widehat{E} \to \widehat{Z}$, $\chi \in \Pic(\ms X_{\widehat{Z}})$ and $g: \ms X_Y \to \ms X$ be as above. Then there is a infinitesimal thickening of $E_{red} \subset E_r$ satisfying the property that if a vector bundle $V$ on $Y$ is trivial on $E_r$ then it is trivial on $\widehat{E}$ and therefore $f_*V$ is a vector bundle on $X$. Moreover, if a twisted vector bundle $V$ on $\ms X_Y$ becomes isomorphic to $\chi^{\oplus k}|_{\ms X_{E_r}}$ upon restriction to $\ms X_{E_r}$ then they also become isomorphic on $\ms X_{\widehat{E}}$ and therefore $g_*V$ is a twisted vector bundle on $\ms X$. If $\dim E \leq 1$ we may also choose an $E_r$ with no embedded associated primes. 
\end{lemma}

\begin{proof} By \cite[Corollary 2.3]{PerlingToric} there is an infinitesimal neighborhood $E_r$ of the exceptional locus of $f$ so that if a vector bundle $E$ on $Y$ is trivial on $E_r$ then it is trivial on $\widehat{E}$ and therefore descends to a vector bundle on $X$. Note that even though their result stated for schemes the same proof works when $X$ is an algebraic space (see \cite[Remark 3.6]{PerlingToric}). Now suppose that a twisted vector bundle $V$ on $\ms X_Y$ becomes isomorphic to $\chi^{\oplus k}|_{\ms X_{E_r}}$ upon restriction to $\ms X_{E_r}$. Then the untwisted vector bundle 
\[V|_{\ms X_{\widehat{E}}} \otimes (\chi|_{\ms X_{\widehat{E}}})^{\vee}\]
on $\widehat{E}$ becomes trivial on $E_r$ and therefore it is trivial on $\widehat{E}$. Thus, we obtain the isomorphism $V|_{\ms X_{\widehat{E}}} \cong \chi|_{\ms X_{\widehat{E}}}^{\oplus k}$ and by Lemma \ref{twisteddescent} we obtain the desired result. Lastly, consider the closed subscheme $E_r' \subset E_r$ realizing the $S_1$-ization of $E_r$. If $\dim E_r \leq 1$ it is defined by an ideal supported at finitely many points of $E_r$ (see \cite[Tag 00RG]{stacks}) and therefore $V|_{E_r'}$ is trivial iff $V|_{E_r}$ is trivial (see \cite[Remark 1.4]{PerlingToric}) so we may replace $E_r$ with $E_r'$. Since $E_r$ is $S_1$, it cannot contain embedded associated points and we are done. \end{proof}

\begin{lemma} \label{closure} Let $X$ be an algebraic space over $k$, if $\mathrm{Br}(X \times_k \bar{k})=\mathrm{Br}'(X \times_k \bar{k})$ then $\mathrm{Br}(X)=\mathrm{Br}'(X)$.
\end{lemma}

\begin{proof} We may write $\bar{k}=\colim k'$ for all finite extensions $k'$ over $k$. Now suppose $\ms X$ is a torsion $\mathbf{G}_m$-gerbe over $X$ and by hypothesis $\ms X \times_k k'=\lim_{k'/k} \ms X \times_k k'$ admits a twisted vector bundle. This implies there is a finite extension $k'/k$ so that $\ms X \times_k k'$ admits a nonzero twisted vector bundle. Since $\ms X \times_k k' \to \ms X$ is a finite flat morphism, the pushforward of this nonzero twisted vector bundle is one as well. As such, the associated cohomological Brauer class of $\ms X$ is geometric. \end{proof}

\begin{lemma} \label{finitedescent} Suppose $X' \to X$ is a finite birational morphism of algebraic spaces of finite-type over $k$ which is an isomorphism away from finitely many points in X. If $\mathrm{Br}(X')=\mathrm{Br}'(X')$ then $\mathrm{Br}(X)=\mathrm{Br}'(X)$.
\end{lemma}

\begin{proof} By lemma \ref{closure} we may assume that $k$ is algebraically closed. Denote by $U \subset X$ the maximal open subset where $X' \to X$ is not an isomorphism and let $Z=X\backslash U$ denote its (finite) complement and $Z^{sh}$ is strict henselization. Let $\ms X$ denote a $\mu_n$-gerbe on $X$ and $\ms X'=\ms X \times_X X'$. By assumption there is a nonzero twisted vector bundle $V'$ on $\ms X'$. Set $Z^{\circ}=Z^{sh} \times_X U$, $Z'=Z^{sh} \times_X X'$ and note that the latter is the disjoint union of strictly henselian rings by \cite[18.5.10]{EGAIV4} and the fact that the residue fields of $Z^{sh}$ are algebraically closed. Our goal is to find a twisted vector bundle $V_U$ over $\ms X_U$, $V_{Z^{sh}}$ over $\ms X_{Z^{sh}}$ and an isomorphism $\phi: V_{Z^{sh}} \to V_U$ over $\ms X_{Z^{\circ}}$. Indeed, such a triple uniquely defines a nonzero twisted vector bundle on $\ms X$ by \cite[Theorem B(1)]{hall2016mayer}. Since $V'$ is a twisted vector bundle on $\ms X'$ this means we have a twisted vector bundle $V_U=V'|_{\ms X' \times_{X'} U}$ over $\ms X_U=\ms X'\times_{X'} U$ since $X' \to X$ is an isomorphism over $U$. Moreover, we have a twisted vector bundle on $V_{Z'}$ over $Z'$ along with an isomorphism $\phi': V_{Z'}|_{\ms X' \times_{X'} Z^{\circ}} \to V_{U}|_{\ms X' \times_{X'} Z^{\circ}}$. Let $\chi$ be a twisted line bundle on $\ms X_{Z^{sh}}$ and $\chi'$ its restriction to $\ms X_{Z'}$ and note that $V_{Z'} \cong \chi'^{\oplus n}$. Indeed, $V_{Z'} \otimes \chi'^{\vee} \cong \mathcal{O}^{\oplus n}_{Z'}$ since it is untwisted and $Z'$ is the spectrum of a finite product of strictly henselian local rings. So if we combine this observation with the above isomorphism $\phi'$ and set $V_Z=\chi^{\oplus n}$ we get the desired isomorphism 
\[\phi: V_Z|_{\ms X \times_X Z^{\circ}} \to V_U|_{\ms X \times_X Z^{\circ}}\]
Thus there is a nonzero twisted vector bundle on $\ms X$ and the claim follows. \end{proof}

\subsection{A Twisted Elementary Transformation}

\begin{proof}[\textbf{Proof of Theorem \ref{T2}}] We may suppose that $k$ is algebraically closed. Now consider the Stein factorization of the morphism $Y \to X' \to X$ and observe that $X' \to X$ is a finite birational morphism which is an isomorphism away from finitely many points. Thus, if $\mathrm{Br}(X')=\mathrm{Br}'(X')$ we may conclude by Lemma \ref{finitedescent}. Thus, by replacing $X$ with $X'$ we may suppose that $f_*\mathcal{O}_Y=\mathcal{O}_X$. 

Let $\ms X$ denote a $\mu_n$-gerbe over $X$ and suppose $\ms X_Y$ denotes its pullback to $Y$. By Lemma \ref{infinitesimal} there is a infinitesimal neighborhood of $E_{\text{red}} \subset E_r$ containing no embedded associated primes with the additional property that if a twisted vector bundle $V$ on $\ms X_Y$ is isomorphic to $\chi^{\oplus k}|_{\ms X_{E_r}}$ when restricted to $\ms X_{E_r}$, then $f_*V$ is a twisted vector bundle on $\ms X$.

By a result of Gabber (see \cite{dejongample}) there is a twisted vector bundle $V$ on $\ms X_Y$, our aim is to modify $V$ so that it descends to a twisted vector bundle on $\ms X$ by Lemma \ref{twisteddescent}. First, replace $V$ with $V^{\oplus n}$ so that the determinant of $V$, call it $L'$, is the pullback of a line bundle on $Y$, let $l$ denote the rank of the new vector bundle $V$. By taking further direct sums of $V$ we may suppose $l$ is larger than the dimension of $X$.

Consider the twisted vector bundle $V|_{\ms X_{E_r}}$ and let $\chi|_{\ms X_{E_r}}$ denote the restriction of the twisted line bundle on $\ms X_{Z^{sh}}$ to $\ms X_{E_r}$. This twisted line bundle has the property that $\chi|_{\ms X_{E_r}}^{\otimes n}\cong \mathcal{O}_{E_r}$. Set $G=V|_{\ms X_{E_r}} \otimes \chi|_{\ms X_{E_r}}^{-1}$ and note that since $G$ is a vector bundle on $E_r$ with rank divisible by $n$
\[\det(G)=\det(V|_{\ms X_{E_r}} \otimes \chi^{-1}|_{\ms X_{E_r}})=\det(V|_{\ms X_{E_r}})=L'|_{E_r}\]
so the determinant of $G$ is the restriction of a line bundle on $Y$. Observe that the hypothesis of Lemma \ref{E-ample} is satisfied and so there exists an ample line bundle $L'' \in \Pic(Y)$ satisfying its conclusion. In other words, by replacing $V$ with $V \otimes L''$, Lemma \ref{E-ample} implies we may assume $G$ is globally generated, and whenever $\det(G) =\mathcal{O}_{E_r}(D)$ for some effective Cartier divisor $D \subset E_r$ not containing any associated point of $Y$, then there is an ample Cartier divisor $H \subset Y$ with $H \cap E_r=D$ scheme-theoretically. Note that $\det(G)$ is still the restriction of a line bundle $L \in \Pic(Y)$.

By Lemma \ref{trivialelementary} there is a finite Cartier divisor $D \subset E_r$ not containing any associated point of $Y$ and an exact sequence of sheaves on $E_r$:
\[0 \to \mathcal{O}_{E_r}^{\oplus l} \to G \to \mathcal{O}_D \to 0\]
Taking determinants we learn that $L|_{E_r}=\det(G)=\mathcal{O}_{E_r}(D)$ so that there exists an effective Cartier divisor $H \subset Y$ with $H \cap E_r=D$ scheme-theoretically. By twisting this exact sequence by $\chi|_{\ms X_{E_r}}$ we obtain an elementary transformation of $V|_{\ms X_{E_r}}$:
\[0 \to \chi|_{\ms X_{E_r}}^{\oplus l} \to V|_{\ms X_{E_r}} \to \chi|_{\ms X_{E_r}} \otimes \mathcal{O}_D \to 0\]
By summing these together, we obtain
\[0 \to \chi^{\oplus l^2} \to V|_{\ms X_{E_r}}^{\oplus l} \to (\chi|_{\ms X_{E_r}} \otimes \mathcal{O}_D)^{\oplus l} \to 0\]
Denote the corresponding surjection $V|_{\ms X_{E_r}}^{\oplus l} \to (\chi|_{\ms X_{E_r}} \otimes \mathcal{O}_D)^{\oplus l}$ by $\phi_D$ and observe that $\chi^{\oplus l^2}$ is an elementary transformation of $V|_{\ms X_{E_r}}^{\oplus l}$ along the Cartier divisor $D \subset E_r$. Our goal is to extend this elementary transformation of $V|_{\ms X_{E_r}}^{\oplus l}$ to one of $V^{\oplus l}$ along the Cartier divisor $H \subset Y$. More precisely, we produce a surjection 
\[\psi: V^{\oplus l} \to V|_{\ms X_H} \otimes L_1^{\otimes N}\]
for some ample line bundle $L_1 \in \Pic(H)$ so that the restriction, $\psi_D$, to $\ms X_{E_r}$ is compatible with $\phi_D$. By this we mean that there is an isomorphism $g: V|_{\ms X_D} \otimes L_1^{\otimes N}|_D \to (\chi|_{\ms X_D} \otimes \mathcal{O}_D)^{\oplus l}$ making the diagram commute

\begin{center} \begin{tikzcd} 
V|_{\ms X_{E_r}}^{\oplus l} \arrow{d}{\text{id}} \arrow{r}{\psi|_{\ms X_{E_r}}} & (V \otimes L_1^{\otimes N})|_{\ms X_D} \arrow{d}{g} \arrow{r} & 0 \\
  V|_{\ms X_{E_r}}^{\oplus l} \arrow{r}{\phi_D} & (\chi|_{\ms X_{E_r}} \otimes \mathcal{O}_D)^{\oplus l} \arrow{r} & 0 \\
\end{tikzcd} \end{center}

By Lemma \ref{Bertini} with $W_1=V^{\oplus l}|_{\ms X_H}$ and $W_2=V|_{\ms X_H}$ we may find a $\psi: V^{\oplus l} \to V|_{\ms X_H} \otimes L_1^{\otimes N}$ with the aforementioned properties. Now observe that this induces an exact sequence of twisted sheaves on $\ms X_Y$:
\[0 \to V' \to V^{\oplus l} \to V|_{\ms X_H} \otimes L_1^{\otimes N} \to 0\]
and note that $V'$ is locally free because $V|_{\ms X_H} \otimes L_1^{\otimes N}$ has projective dimension $1$. Indeed, $V|_{\ms X_H} \otimes L_1^{\otimes N}$ is a vector bundle on the Cartier divisor $\ms X_H \subset \ms X_Y$. We claim that $f_*V'$ is a twisted vector bundle on $\ms X$. Indeed, upon restriction to $\ms X_{E_r}$ the sequence remains exact (since $H$ and $E_r$ share no components and $E_r$ has no embedded associated points) and in fact, the isomorphism $g$ induces an isomorphism of short exact sequences:

\begin{center} \begin{tikzcd} 
0 \arrow{r} & V'|_{\ms X_{E_r}} \arrow{r} \arrow{d}{\cong} & V|_{\ms X_{E_r}}^{\oplus l} \arrow{d}{\text{id}} \arrow{r}{\psi|_{\ms X_{E_r}}} & (V \otimes L_1^{\otimes N})|_{\ms X_D} \arrow{d}{g} \arrow{r} & 0 \\
 0 \arrow{r} & \chi^{\oplus l^2}|_{\ms X_{E_r}} \arrow{r} & V|_{\ms X_{E_r}}^{\oplus l} \arrow{r}{\phi_D} & (\chi|_{\ms X_{E_r}} \otimes \mathcal{O}_D)^{\oplus l} \arrow{r} & 0 \\
\end{tikzcd} \end{center}
Since $V'|_{\ms X_{E_r}} \cong \chi|_{\ms X_R}^{\oplus l^2}$, Lemma \ref{twisteddescent} implies $f_*V'$ is a twisted vector bundle on $\ms X$.
 \end{proof}

\begin{proof}[\textbf{Proof of Corollary \ref{surfaces}}] We need to show that $X$ satisfies the hypothesis of Theorem \ref{T2}. By \cite[Tag 0ADD]{stacks} there is a schematic open subset $U \subset X$ containing all but finitely many closed points of $X$. Now by \cite[Theorem 1.5]{SurfaceRes} there is a quasi-projective open subset $V \subset U$ containing all but finitely many closed points of $X$. Then by Chow's lemma as in \cite[Tag 088R]{stacks}, there is a $V$-admissible blow-up $Y \to X$ with $Y$ quasi-projective over $k$. In particular, $Y \to X$ is a proper birational morphism which is an isomorphism away from finitely many points in $X$. Moreover, since $X$ is $2$-dimensional, the fibers of the morphism $Y \to X$ all have dimension $\leq 1$. The result follows. \end{proof}

\bibliography{mybib}{}
\bibliographystyle{plain}

\end{document}